\documentclass[manuscript,screen,authorversion,nonacm]{acmart}

\usepackage{listings}
\usepackage{pdflscape}
\usepackage{epsfig}
\usepackage{lscape}
\usepackage{longtable}
\usepackage{amsmath,amsthm,amsfonts}
\usepackage{graphicx}
\usepackage{color}
\usepackage{placeins}
\usepackage{url}
\usepackage{cases}
\usepackage{array}
\usepackage{multirow, multicol, makecell}
\usepackage{natbib, hyperref}

\newcommand{\blue}[1]{#1} 

\DeclareMathOperator*{\argmin}{argmin}

\def\norm#1{\|#1 \|}
\def\inprod#1#2{\langle#1,\,#2 \rangle}

\def\cA{\mathcal{A}}  
\def\cC{\mathcal{C}}  \def\cP{\mathcal{P}} \def\cT{\mathcal{T}}
  \def\cG{\mathcal{G}} 
 \def\cM{\mathcal{M}} \def\cN{\mathcal{N}} \def\cW{\mathcal{W}}
\def\cS{\mathcal{S}} \def\cE{\mathcal{E}}  \def\cV{\mathcal{V}}
 \def\cX{\mathcal{X}} \def\cY{\mathcal{Y}} 
 \def\cN{\mathcal{N}} \def\cL{\mathcal{L}}
   
\def\Sn{\mathcal{S}^n}
\def\svec{{\rm {\bf svec}}}   

\def\vec{{\rm {\bf vec}}}     \def\diag{{\rm Diag}}

\def\Range{\textup{Range}}
\def\Diag{\textup{Diag}}

\def\sig{\sigma}

\def\by{\bar{y}}

\def\SGS{\mbox{sGS}}

\newlength{\len}
 \newtheorem{prop}{Proposition}[section]
 \newtheorem{remark}{Remark}[section]
 \newtheorem{assumption}{Assumption}[section]

\newcolumntype{H}{>{\setbox0=\hbox\bgroup}c<{\egroup}@{}}

\AtBeginDocument{%
  \providecommand\BibTeX{{%
    \normalfont B\kern-0.5em{\scshape i\kern-0.25em b}\kern-0.8em\TeX}}}

\begin{document}

\title{QPPAL: A two-phase proximal augmented Lagrangian	method for high dimensional convex quadratic programming problems}

\renewcommand{\shorttitle}{QPPAL: A two-phase proximal ALM for convex QPs}

\author{Ling Liang}
\affiliation{%
  \institution{Department of Mathematics, National University of Singapore}
  \streetaddress{10 Lower Kent Ridge Road}
  \country{Singapore}
  \postcode{119076}}
\email{email:liang.ling@u.nus.edu}

\author{Xudong Li}
\authornote{The research of this author is supported by the National Key R\&D Program of China 2020YFA0711900, 2020YFA0711901, the National Natural Science Foundation of China (11901107), and the Young Elite Scientists Sponsorship Program by CAST (2019QNRC001).}
\affiliation{%
  \institution{{School of Data Science}, Fudan University}
  \city{Shanghai}
  \country{China}}
\email{email:lixudong@fudan.edu.cn}

\author{Defeng Sun}
\authornote{The research of this author is supported by NSFC/RGC Joint Research Scheme under Grant {\rm ${\rm N}_{-}$}PolyU504/19.}
\affiliation{%
  \institution{Department of Applied Mathematics, the Hong Kong Polytechnic University}
  \streetaddress{Hung Hom}
  \country{Hong Kong}
}
\email{email:defeng.sun@polyu.edu.hk}

\author{Kim-Chuan Toh}
\authornote{The research of this author is supported by the Ministry of Education, Singapore, under its Academic Research Fund Tier 3 grant call (MOE-2019-T3-1-010).}
\affiliation{%
 \institution{Department of Mathematics and Institute of Operations Research and Analytics, National University of Singapore}
 \streetaddress{10 Lower Kent Ridge Road}
 \country{Singapore}
 \postcode{119076}}
\email{email:mattohkc@nus.edu.sg}

\renewcommand{\shortauthors}{Liang et al.}

\begin{abstract}
  In this paper,  we aim to solve high dimensional convex quadratic programming (QP) problems with a large number of quadratic terms, linear equality and inequality constraints. In order to solve the targeted {\bf QP} problems to a desired accuracy efficiently, we develop a two-phase {\bf P}roximal {\bf A}ugmented {\bf L}agrangian method {(QPPAL)}, with Phase I to generate a reasonably good initial point to warm start Phase II to obtain an accurate solution efficiently. More specifically, in Phase I, based on the recently developed symmetric Gauss-Seidel (sGS) decomposition technique, we design a novel sGS based semi-proximal augmented Lagrangian method for the purpose of finding a solution of low to medium accuracy. Then, in Phase II, a proximal augmented Lagrangian algorithm is proposed to obtain a more accurate solution efficiently. Extensive numerical results evaluating the performance of {QPPAL} against {existing state-of-the-art solvers Gurobi, OSQP and QPALM} are presented to demonstrate the high efficiency and robustness of our proposed algorithm for solving various classes of large-scale convex QP problems. {The MATLAB implementation of the software package QPPAL is available at: \url{https://blog.nus.edu.sg/mattohkc/softwares/qppal/}.}
\end{abstract}

%


\maketitle

\section{Introduction}\label{intro}
We begin with some notation that will be used throughout the paper. Let $\mathcal{S}_+^n$ be the cone of $n\times n$ symmetric and positive semidefinite matrices in the space of $n\times n$ symmetric matrices $\mathcal{S}^n$ endowed with the standard trace inner product $\inprod{\cdot}{\cdot}$ and the Frobenius norm $\norm{\cdot}$. The range space of a matrix $Q\in \cS^n$ is denoted by $\mathrm{Range}(Q)$. Let $\cX$ be any real finite dimensional Euclidean space and $\cM:\cX \rightarrow \cX$ be any self-adjoint positive semidefinite linear operator, denote $\|x\|_\cM  = \sqrt{\inprod{x}{\cM x}}$ and $\mathrm{dist}_{\cM}(x, C) = \inf _{x'\in C} \norm{x' - x}_{\cM}$ for any $x\in \cX$ and any set $C\subseteq \cX$. For a given closed proper convex function $\theta:\Re^n\rightarrow (-\infty, +\infty]$, the effective domain of $\theta$ is defined by $\mathrm{dom}\,\theta = \{x\in \Re^n:\theta(x)<\infty\}$, the subdifferential of $\theta$ at $x\in \mathrm{dom}\, \theta$ is defined by $\partial \theta (x) = \{v\in \Re^n:\theta(y)\geq \theta (x) + \inprod{v}{y - x},\; \forall y\in \Re^n\}$ and the convex conjugate function $\theta^*:\Re^n\rightarrow (-\infty, +\infty]$ of $\theta$ is defined by $\theta^*(y):= \sup \{ \inprod{y}{x} - \theta(x):x\in \Re^n\} $. {In particular, for a given closed convex set $C\subseteq \cX$, we use $\delta_C(\cdot)$ to denote the indicator function of $C$ {with $\delta_C^*(\cdot)$ being  its convex conjugate function}. Moreover, the relative interior of $C$ is denoted by $\mathrm{ri}(C)$}. For more details on convex sets and convex functions, we recommend the monograph \cite{rockafellar1970convex}.

Consider the high-dimensional convex quadratic programming (QP) problem in the following standard form:
\begin{equation*}
    ({\bf P}) \quad  \min_{x\in \Re^n} \left\{\frac{1}{2}\inprod{x}{Qx} + \inprod{c}{x}
	\mid A x=  b, \; x \in \cC \right\},
\end{equation*}
{where $c\in \Re^n$, $ Q \in \Sn_+$, $A \in \Re^{m\times n}$, $b\in \Re^m$, {$\cC = \left\{x\in\Re^n \; : l \le x \le u\right\}$ with the given vectors {$\ell, u \in \Re^n$} satisfying $-\infty \leq l\leq u\leq +\infty$}. Note that $ {\bf (P)}$ covers convex QP problems with linear inequality constraints by adding slack variables. However, we only consider $ {\bf (P)} $ in the theoretical development for the purpose of notational simplicity. \blue{Moreover, we assume for the rest of this paper that problem {\bf (P)} satisfies the Slater's condition.}
	
In this paper, we are interested in the case where the} dimensions $n$ and/or $m$ are extremely large. Since $n$ is huge, one generally cannot expect an explicitly matrix representation for $Q$. Even if it is available, one may encounter severe memory issues when trying to store a large-scale and dense matrix $Q$.  Hence, in this paper, we only assume that $Q$ is defined as a linear operator on $\Re^n$, and its matrix representation is not needed explicitly, i.e., for any given $x\in\Re^n$, $Qx$ can be obtained at a reasonable cost but the matrix representation of $Q$ with respect to the standard basis in $\Re^n$ may not be available.

{The} {restricted-Wolfe} dual \cite{Li2018QSDPNAL} of ({\bf P}) can be written in the form of
\begin{equation*}
    ({\bf D}) \quad    \max_{{(z, w, y)\in \Re^n\times\Re^n\times\Re^m} } \left\{-\delta_{\cC}^*(-z) -\frac{1}{2}\inprod{w}{Qw} + \inprod{b}{y}
	\mid z - Qw  + A^* y = c, \;
	w\in\mathcal{W}\right\},
\end{equation*}
where $\cW$ is any subspace of $\Re^n$ containing $ \textup{Range}(Q)$. In this paper, we fix $\cW = \Range(Q)$. We will see in the subsequent analysis that this choice in fact plays an important role in the design of our algorithms. Problem ({\bf D}) belongs to a general class of  multi-block convex composite quadratic optimization problems of the form:
\begin{equation}
    \label{eq-dd-M}
	\min_{{y_i\in \cY_i,i = 1,\dots, p}} \Big\{ \theta(y_1) + f(y_1,y_2,\ldots,y_p) \,\mid\,
		\cA_1^*y_1 + \cA_2^*y_2 + \cdots + \cA_p^*y_p = c \Big\},
\end{equation}
where $p$ is a given positive integer, $\theta:\cY_1\rightarrow(-\infty,+\infty]$ is a closed proper convex function whose proximal mapping is assumed to be computable at a moderate cost, $f:\cY_1\times\cY_2\times\ldots\times\cY_p\rightarrow \Re$ is a convex quadratic function (not necessarily separable), $\cA_i:\cX\rightarrow \cY_i, \, i=1,\ldots,p$ are linear maps, $\cY_1, \ldots, \cY_p$ and $\cX$ are all real finite dimensional Euclidean spaces each equipped with an inner product $\inprod{\cdot}{\cdot}$ and its induced norm $\norm{\cdot}.$ For notational convenience, we let $\cY := \cY_1\times\cY_2\times,\ldots,\cY_p$, and write $y {:=} (y_1,y_2,\ldots,y_p)\in \cY $. Moreover, define the linear map $\cA:\cX \rightarrow \cY$ whose adjoint map is given by $\cA^*y = \sum_{i=1}^p\cA_i^*y_i,\; \forall y\in \cY$.

Convex QP has been extensively studied for the last few decades, see, for example the influential works \cite{OSQP,dantzig1961quadratic,Dant:1963,sun1993convergence,gould2000quadratic,gould2001solution,gould1985practical,cottle1964note,cottle1962symmetric,wright1999numerical,ye1998complexity} and references therein. One may also refer to the QP webpage\footnote{\url{http://www.numerical.rl.ac.uk/people/nimg/qp/qp.html}} for more information. {To the best of our knowledge, all the major software packages for solving convex QP problems are based on active set methods \cite[Chapter 16.4]{wright1999numerical}, interior point methods \cite{IPMNesterov}, operator splitting methods \cite{OSQP} or primal based (proximal) augmented Lagrangian methods \cite{hermans2019qpalm, hermans2020qpalm}}. Among these methods, active set methods have the appealing feature to drop many of the inactive constraints to make the problem smaller in scale and hence much easier to solve. However, the worst-case iteration complexity of active set methods can be exponentially large with respect to the problem size and they may take a long time to solve the problem when the active sets are not estimated correctly. Achieving great progress  over the past few decades, interior point method based solvers are perhaps the most notable ones for solving large-scale convex QPs problems. For example, as a representative interior point method based solver, Gurobi \cite{gurobiv6.0}\footnote{Based on the results presented in \url{http://plato.asu.edu/ftp/barrier.html}} is a highly optimized state-of-the-art solver for large-scale convex QP problems and is often used as a computational backbone of many real world applications. However, for solving high dimensional convex QP problems with a large number of constraints, interior point method based solvers (e.g., Gurobi) may encounter inherent numerical difficulties. Indeed, the computational costs of these methods become prohibitively expensive when the systems of linear equations to be solved are fully dense or when the corresponding sparse Cholesky factors are dense. Unlike interior point methods which are generally considered as second-order methods, first-order methods such as operator splitting methods (including alternating direction methods of multipliers) have been at the forefront of the recent progress in solving convex optimization problems. For example, a well-known operator splitting algorithm for solving convex QP problems is the open source solver OSQP studied in \cite{OSQP}. First-order methods have the appealing feature that the per-iteration cost is quite cheap and hence they are highly scalable. However, these methods generally can only return approximate solutions with low to medium accuracy and they often stagnate even before delivering a crude approximate solution. Therefore, if more accurate solutions are needed, first order methods may not be sufficient. {Recently, a highly optimized and efficient solver, named QPALM, was {developed by applying} the proximal ALM method to the primal problem {\bf (P)}; see \cite{hermans2019qpalm, hermans2020qpalm}. QPALM is demonstrated to be quite efficient for solving many sparse QP problems. However, based on our experience gained from many numerical tests, QPALM is observed to be much less efficient under the high-dimensional settings considered in the present paper.} Lastly, as far as we are aware of, the major solvers just mentioned and their variants all require an explicit matrix representation of $Q$. Thus, there is clearly a need to design an algorithm which can handle high dimensional convex QP problems beyond the scope covered by highly optimized solvers such as {Gurobi, OSQP and QPALM}.

We shall next raise the following question: Can we design a highly efficient, scalable and robust algorithm for solving convex QP problems having the following three characteristics? (a) the matrix representation of $Q$ may not be available; (b) $Q$ does not have {a} favourable sparsity pattern; (c) the number of linear constraints is extremely large or there are many dense linear constraints. We try to provide a positive answer to the above question by embracing the influential {(dual based)} augmented Lagrangian method (ALM) for solving the more general problem \eqref{eq-dd-M}. In our opinion, the {(dual based)} ALM is perhaps the most promising algorithm for problem \eqref{eq-dd-M} which has some of or all the three characteristics just mentioned. To briefly explain the idea of ALM, let $\sigma > 0$ be a given parameter and the augmented Lagrangian function associated with \eqref{eq-dd-M} is defined by
\begin{equation*}
	\cL_\sigma(y;x):= \theta(y_1) + f(y) + \inprod{x}{\cA^*y - c}+ \frac{\sigma}{2} \norm{\cA^*y - c}^2,\quad {(x,y)\in \cX\times \cY.}
\end{equation*}
Starting with any initial points $y^0 \in {\rm dom}(\theta)\times\cY_2\times\ldots\times\cY_p $ and $x^0 \in \cX$, ALM performs the following steps iteratively:
\begin{align}
	y^{k+1} = &\; \textup{argmin}_{{y\in \cY}}
	\;\cL_\sigma(y;x^k),
	\label{eq-auglagl-y}
	\\
	x^{k+1} =& \; \displaystyle  x^k +\tau \sigma (\cA^*y^{k+1}  -c),
	\label{eq-auglagl-x}
\end{align}
where $ \tau\in(0,2)$ guarantees the convergence. However, in the high dimensional setting, the nonseparable quadratic terms and the composite structure in the inner subproblem \eqref{eq-auglagl-y} make the task of computing $y^{k+1}$ exactly or with high accuracy extremely difficult and expensive. Fortunately, this difficulty could be alleviated if a good initial point is provided for the ALM, in light of the experience gained from a series of works \cite{SDPNAL+, Li2018QSDPNAL, Li2019LP, SDPNAL} on developing elegant theoretical properties and efficient implementation of the (proximal) ALM for solving several classes of optimization problems. In fact, the ALM equipped with a semismooth Newton method for solving the ALM inner subproblems is shown to be a highly efficient approach for solving \eqref{eq-dd-M} to a high accuracy, if the initial iterate lies in the fast convergence region of the semismooth Newton method. In this paper, by further exploring the idea in the Schur complement based semi-proximal alternating direction method of multipliers (ADMM) proposed in the recent papers~\cite{LiSunToh_indefinite2014, scbADMM2016}, we are able to propose a symmetric Gauss-Seidel based semi-proximal ALM to efficiently solve the non-separable convex composite optimization problem \eqref{eq-dd-M} to low {\em or} medium accuracy. Therefore, we shall use this algorithm as a warm-starting scheme to provide a reasonably good initial point for the ALM. Using this initial point, we then propose a proximal ALM to compute a highly accurate solution efficiently. Consequently, we come up with a two-phase algorithm. {As we shall see later in the numerical experiments, the proposed algorithmic framework is shown to be more suitable for large-scale convex QP problems having the aforementioned characteristics compared to interior point methods, operator splitting algorithms and the primal based ALM. Indeed, we observe that {applying the proximal ALM to the {(restricted-Wolfe)} dual problem {\bf (D)} appears to be} more effective.}


The remaining parts of this paper are organized as follows. In Section \ref{sec-isPALM}, we first propose an inexact semi-proximal augmented Lagrangian method (isPALM) and establish its convergence. Then, as our phase I algorithm for solving the convex composite quadratic programming model \eqref{eq-dd-M}, a symmetric Gauss-Seidel based inexact semi-proximal augmented Lagrangian method (sGS-isPALM) is designed via incorporating the sGS decomposition technique with the aforementioned isPALM algorithm. In Section \ref{sec-qppal}, we propose our two-phase algorithm QPPAL. In QPPAL Phase I, the sGS-isPALM is directly applied to solve the convex quadratic programming problem ({\bf D}). Then, in QPPAL Phase II, a proximal ALM, with the semismooth Newton method for solving the inner minimization problems, is proposed and the convergence are also established. In section \ref{sec-numerical}, we discuss key implementation issues and present numerical experiments to evaluate our QPPAL in solving some classes of large-scale convex QP problems. We conclude our paper in Section \ref{sec-conclusion}.

\section{An inexact semi-proximal augmented Lagrangian method}
\label{sec-isPALM}
In this section, by revisiting the convergence of the inexact semi-proximal ALM and applying the symmetric Gauss-Seidel (sGS) decomposition technique to the convex composite quadratic programming model \eqref{eq-dd-M}, we shall propose an sGS based inexact semi-proximal ALM method with convergence guarantees. {These theoretical results have been analyzed extensively in the recent work \cite{ChenipALM} and references therein. In this work,
 we focus on the {detailed} numerical implementation of the methods for solving convex QPs. Hence, we refrain from adding too much theoretical detail on the algorithm}.

To begin, we first consider the following linearly constrained convex optimization problem
\begin{equation}
    \label{aug-problem}
	\min_{{v\in \cV}} \Big\{ g(v)    \,\mid\, \cG^*v  = c \Big\},
\end{equation}
where $g:\cV \rightarrow (-\infty,+\infty]$ is a closed proper convex function, $\cG: \cX\rightarrow \cV$ is a given linear map and $\cV$ is a real finite dimensional Euclidean space.
We make the following standard {solvability} assumption for \eqref{aug-problem}.
\begin{assumption}
	\label{assumption:KKT}
	The solution set to the following KKT system of \eqref{aug-problem}
	\begin{equation}\label{eq-KKT-g}
		0\in\partial g(\bar v) + \cG\bar x,\quad \cG^*\bar v = c, \quad (\bar{x}, {\bar{v})\in \cX\times \cV}
	\end{equation}
	is nonempty.
\end{assumption}

Let $\sigma > 0$ be a given parameter. The augmented Lagrangian function associated with (\ref{aug-problem}) is given as follows:
\begin{equation*}
	\cL_{\sigma}(v;x) = g(v)
	+\inprod{x}{\cG^*v-c}
	+ \frac{\sigma}{2}\norm{\cG^*v-c}^2,\quad (v,x)\in \cV \times \cX.
\end{equation*}
Let $\partial g$ be the subdifferential mapping of $g$.  Then $\partial g$ is a maximally monotone operator. Hence, there exists a self adjoint positive semidefinite linear operator  $\Sigma_g $   such that  for all  $v,  \tilde{v} \in {\rm {dom}}(g)$, $\zeta\in \partial
g(v )$, and $ \tilde{\zeta}\in \partial g(\tilde{v})$, it holds that (see e.g., \cite[Theorem 12.17]{rockafellar2009variational})
\begin{equation}\label{monosub2}
	\langle \zeta-\tilde{\zeta}, v -\tilde{v}\rangle \geq
	\|v -\tilde{v}\|^2_{\Sigma_g}.
\end{equation}
The inexact semi-proximal augmented Lagrangian method ({isPALM}) for solving \eqref{aug-problem} is described in Fig.~\ref{alg-isalm}.

\begin{figure}[ht!]
	\centerline{\fbox{\parbox{\textwidth}{
				{\bf Algorithm {isPALM}: An inexact semi-proximal augmented Lagrangian method for \eqref{aug-problem}.}
				\\
				Let $\sigma >0$ and $\tau\in(0,2)$ be given  parameters, $\{\varepsilon_k\}_{k\ge 0}$ be a nonnegative summable sequence. Let $\cT$ be a given self-adjoint positive semidefinite linear operator defined on $\cV$ such that $\cN := \Sigma_g + \cT + \sigma\cG\cG^* \succ 0 $. Choose $(v^0, x^0)\in\mbox{dom}(g) \times\cX.$ Perform the following steps in each iteration for {$k = 0, 1, 2\dots,$.}
				\begin{description}
					\item [Step 1.] Compute
					\begin{equation}
						\label{aug-v}
						v^{k+1} \approx{} \bar v^{k+1}:= \; \argmin_{v\in \cV} \;\cL_{\sigma}(v;x^k) + \frac{1}{2}\norm{v - v^k}^2_{\cT}
					\end{equation}
					such that there exists $d_k$ satisfying $\norm{\cN^{-1/2}d^{k}}\leq \varepsilon_k$ and
					\begin{equation}
						\label{aug-d}
						d^{k} \in{} \partial_v\cL_{\sigma}(v^{k+1};x^k) + \cT(v^{k+1}-v^k).
					\end{equation}
					\item [Step 2.] Compute
					$x^{k+1} = x^k + \tau \sigma(\cG^*v^{k+1} -c).$
				\end{description}
	}}}
	\caption{Algorithm isPALM. \label{alg-isalm}}
\end{figure}

The global convergence result for Algorithm {isPALM} under certain technical assumptions is presented as follows whose proof can be taken directly from the one in \cite[Theorem 3.1]{ChenipALM}.

\begin{theorem}\label{thmisPALM}
	Assume that Assumption \ref{assumption:KKT} holds and that $\Sigma_g + \cT + \sigma\cG\cG^*\succ 0$. Let $\{(v^k,x^k)\}$ be generated from Algorithm {isPALM}.
	Then the following results hold:
	\begin{enumerate}
		\item [{\rm (a)}] the sequence $\{(v^k, x^k)\}$ is bounded;
		\item [{\rm (b)}] any accumulation point of the sequence $\{(v^k, x^k)\}$ solve the KKT system of \eqref{aug-problem};
		\item [{\rm (c)}] the whole sequence $\{(v^k, x^k)\}$ converges to a solution to the KKT system of \eqref{aug-problem}.
	\end{enumerate}
\end{theorem}

\subsection{A symmetric Gauss-Seidel based inexact semi-proximal ALM}
In the remaining part of this section, we focus on the convex composite quadratic programming model \eqref{eq-dd-M} where the convex quadratic function $f:\cY\rightarrow \Re$ is defined by $f(y) = \frac{1}{2}\inprod{y}{ \cP y} - \inprod{b}{y}$ with $b\in\cY$ and $\cP$ being a self-adjoint positive semidefinite linear operator defined on $\cY$.

For later discussions, we consider the following decomposition for $\cP$:
\[\cP y \equiv \left(
\begin{array}{cccc}
	\cP_{11} & \cP_{12} & \cdots & \cP_{1p} \\
	\cP_{12}^* & \cP_{22} &\cdots  & \cP_{2p}  \\
	\vdots& \vdots & \ddots & \vdots \\
	\cP_{1p}^*& \cP_{2p}^* & \cdots & \cP_{pp} \\
\end{array}
\right) \left(
\begin{array}{c}
	y_1 \\
	y_2 \\
	\vdots \\
	y_p \\
\end{array}
\right)
,\]
where $\cP_{ij}:\cY_j \rightarrow \cY_i$, $i= 1,\ldots,p,\; j\leq i$ are linear maps.

We first introduce a self-adjoint semidefinite linear operator $\cS_1$ defined on $\cY_1$ to handle the convex, possibly nonsmooth, functions $\theta(y_1)$, such that
\begin{eqnarray}\label{cS1}
	\cE_{11}: = \cP_{11} + \cS_{1} + \sigma \cA_1\cA_1^* \succ 0,
\end{eqnarray}
and the following well-defined optimization problem
\[
 \min_{y_1}\left\{ \theta(y_1) + \frac{1}{2}\norm{y_1-\by_1}^2_{\cE_{11}}\,\mid\, y_1\in \cY_1\right\}
 \]
can easily be solved for any $\by_1\in\cY_1$. Then, for $i=2,\ldots,p$, let $\cS_{i}$ be a self-adjoint positive semidefinite linear operator on $\cY_i$ such that
\begin{eqnarray}
	\cE_{ii} :=  \cP_{ii} + \sigma\cA_i\cA_i^* + \cS_i \succ 0.
	\label{Tyi}
\end{eqnarray}
In practice, we would choose $\cS_{i}$ in  such a way that the inverse of $\cE_{ii}$
can be computed at a moderate cost. But note that for the algorithm under consideration to be efficient, we need  $\cS_i$ to be as small as possible for each $i=1,\ldots,p$.

Now we are ready to present our symmetric Gauss-Seidel based inexact semi-proximal augmented Lagrangian ({sGS-isPALM}) algorithm for solving \eqref{eq-dd-M} in Fig.~\ref{alg-sgsispalm}.

\begin{figure}[ht!]
	\centerline{\fbox{\parbox{\textwidth}{
				{\bf Algorithm {sGS-isPALM}: A symmetric Gauss-Seidel based inexact semi-proximal augmented Lagrangian method for solving \eqref{eq-dd-M}.}
				\\
				Let $\sigma >0$ and $\tau\in(0,2)$ be given  parameters, $\{\epsilon_k\}_{k\ge 0}$ be a nonnegative summable sequence. Choose $(y^0,x^0)\in\mbox{dom}(\theta_1)\times\cY_2\times\ldots\times\cY_p \times\cX.$
				Perform the following steps in each iteration.
				\begin{description}
					\item [Step 1.] (Backward GS sweep) Compute for $i=p,  \ldots, 2,$
					\begin{equation*}
						\overline{y}_i^k \;\approx {} \argmin_{{y_i\in \cY_i}}
						\cL_{\sigma}((y_{< i}^k,y_i,\overline{y}_{> i}^{k});x^k) + \frac{1}{2}\norm{y_i - y_i^k}_{\cS_i}^2 
					\end{equation*}
					such that there exists $\delta_i^k$ satisfying $\norm{\delta_i^k}\leq \epsilon_k$ and
					\begin{equation*}
						\delta_i^k \in{} \partial \cL_{\sigma}((y_{< i}^k,\overline{y}^k_i,\overline{y}_{> i }^{k});x^k)
						+\cS_i(\overline{y}^k_i - y^k_i).
					\end{equation*}
					\item [Step 2.]  (Forward GS sweep) Compute for $i=1,\ldots,p,$
					\begin{equation*}
						y_i^{k+1} \;\approx {} \argmin_{{y_i\in \cY_i}} \cL_{\sigma}((y_{<i}^{k+1},y_i,\overline{y}_{>i}^k);x^k)
						+ \frac{1}{2}\norm{y_i - y_i^k}_{\cS_i}^2
					\end{equation*}
					such that there exists $\hat{\delta}_i^k$ satisfying $\norm{\hat\delta^k_i}\le \epsilon_k$ and
					\begin{equation*}
						\hat\delta_i^k \in {}\partial \cL_{\sigma}((y_{< i}^{k+1},y_i^{k+1},\overline{y}_{>i}^k);x^k)
						+ \cS_i(y_i^{k+1} - y^k_i).
					\end{equation*}
					\item [Step 3.] Compute
					$x^{k+1} = x^k + \tau \sigma( \cA^*y^{k+1}  -c).$
				\end{description}
	}}}
	\caption{Algorithm sGS-isPALM. \label{alg-sgsispalm}}
\end{figure}

In order to prove the convergence of Algorithm {sGS-isPALM} for solving \eqref{eq-dd-M}, we shall study the relationship between Algorithm {sGS-isPALM} and Algorithm {isPALM}. To this end, let $\cS: = \Diag(\cS_1,\ldots,\cS_p)$ and define the following linear operators:
\begin{eqnarray}
	\cE := \cP + \sigma\cA\cA^* + \cS \;=\; \cE_u^* + \cE_d + \cE_u
	\label{eq-cE},\quad
	\SGS(\cE) = \cE_u \cE_d^{-1} \cE_u^*,
\end{eqnarray}
where $\cE_d = \Diag(\cE_{11},\ldots,\cE_{pp})$ and
\begin{eqnarray*}
	\cE_u:=\left( \begin{array}{cccc}
		{\bf 0} & \cP_{12}+\sigma\cA_1\cA_2^* & \cdots & \cP_{1p}+\sigma\cA_1\cA_p^* \\
		& \ddots &\cdots  & \vdots  \\
		&  & {\bf 0} & \cP_{(p-1)p}+\sigma\cA_{p-1}\cA^*_p \\[8pt]
		&  &  & {\bf 0} \\
	\end{array}
	\right).
\end{eqnarray*}

For $k\ge 0$, let $\delta_1^k = \hat\delta_1^k$, $\delta^k :=(\delta_1^k,\ldots,\delta_p^k)$ and $\hat\delta^k:=(\hat\delta_1^k,\ldots,\hat\delta_p^k)$. \blue{Then we} have the following result which establishes the relationship between Algorithm {sGS-isPALM} and Algorithm {isPALM}. We refer the readers to Appendix \ref{subsec-pf-padmm}.
\begin{prop}\label{prop:equi-scb-padmm}
	It holds that $\widehat\cE = \cE+\SGS(\cE) \succ0$. For any $k\ge 0$, the point $(y^{k+1}, x^{k+1})$ obtained by Algorithm {sGS-isPALM} for solving problem \eqref{eq-dd-M} can be generated exactly according to the following iteration:
	\begin{equation}\label{uykp1}
		\left\{
		\begin{aligned}
			y^{k+1} ={}& \argmin_{{y\in \cY}}
			\cL_{\sigma}(y;x^k) + \frac{1}{2}\norm{y-y^k}^2_{\cS + \SGS(\cE)} - \inprod{d^k}{y} \\
			x^{k+1} ={}& x^k + \tau \sigma(\cA^*y^{k+1} -c),
		\end{aligned}
		\right.
	\end{equation}
	where $d^k = \hat{\delta}^k + \cE_u\cE_d^{-1}(\hat\delta^k - \delta^k)$. Moreover, it holds that
	\[\norm{\widehat\cE^{-1/2}d^k}\le \big((2p-1)\norm{\cE_d^{-1/2}}+p\norm{\widehat\cE^{-1/2}}\big)\epsilon_k,\quad \forall k\ge 0.\]
\end{prop}

By combing Theorem \ref{thmisPALM} with Proposition \ref{prop:equi-scb-padmm}, we can finally state our main convergence theorem under suitable assumptions.

\begin{theorem}\label{thmsGSisPALM}
	Suppose that the solution set of problem \eqref{eq-dd-M} is nonempty and that  there exists $\hat y \in {\rm ri}({\rm dom}\,\theta) \times \cY_2 \times\ldots\times\cY_p $ such that $\cA^* \hat y =c$. Let $\{(y^k,x^k)\}$ be generated from Algorithm {sGS-isPALM} with $\tau\in(0,2)$. Then, the sequence $\{y^k\}$ converges to an optimal solution to problem \eqref{eq-dd-M} and $\{x^k\}$ converges to an optimal solution of the dual of problem \eqref{eq-dd-M}.
\end{theorem}
\begin{remark}
	\label{rmk-nonergodic}
	We can also establish the nonergodic iteration complexity for the sequence generated by Algorithm sGS-isPALM. For more details on this topic, we refer the readers to \cite[Theorem 4.4]{ChenipALM}.
\end{remark}

\section{A two-phase proximal ALM for solving convex QP problems}\label{sec-qppal}
In this section, we shall present a two-phase proximal ALM for solving convex quadratic programming problems ({\bf D}) to high accuracy efficiently. For simplicity, we call our algorithm QPPAL. We will discuss the implementation detail for the algorithm in Section \ref{sec-implementation}. 

\subsection{QPPAL Phase I}
In Phase I, we shall apply Algorithm {sGS-isPALM} directly to solve ({\bf D}). Given $\sigma>0$, let $L_{\sigma}(z,w,y;x)$ be the augmented Lagrangian function associated with problem ({\bf D}) (here we {reformulate} ({\bf D}) as a minimization problem), i.e., for any $(z,w,y,x)\in\Re^n\times\cW\times\Re^m\times\Re^n$,
\begin{equation}\label{eq-aug-d}
	L_{\sigma}(z,w,y;x) = \delta_{\cC}^*(-z) + \frac{1}{2}\inprod{w}{Qw} - \inprod{b}{y}
	+\frac{\sigma}{2}\norm{z-Qw+A^*y-c+\sig^{-1}x}^2
	-\frac{1}{2\sig}\norm{x}^2.
\end{equation}
Then, the detailed steps of our Phase I algorithm for convex quadratic programming are given in Fig.~\ref{alg-qppal-phase-1}.

\begin{figure}[ht!]
	\centerline{\fbox{\parbox{\textwidth}{
				{\bf Algorithm {QPPAL-Phase-I}}: {\bf An {sGS-isPALM} method for ({\bf D}).}
				\\
				Select an initial point  $(z^{0},w^0,y^{0})$ with $-z^0\in
				\textup{dom}(\delta^*_\cC)$, $(w^0,y^0)\in\cW\times\Re^m$. Let $\{\epsilon_k\}$ be a summable sequence of
				nonnegative numbers, $\sigma >0$ and $\tau\in(0,2)$ be given  parameters.
				Set $k=0$. Iterate the following steps.
				\begin{description}				
					\item[Step 1.]  Compute
					\begin{eqnarray*}
						\begin{array}{lll}
							\by^{k} &=& \argmin_{y}\left\{L_{\sigma}(z^k,w^k,y;x^k) - \inprod{\delta_E^k}{y}\mid y\in \Re^m\right\},\\
							\bar w^{k} &=& \argmin_{w}\left\{L_{\sigma}(z^k,w,\by^k;x^k) - \inprod{{\delta_Q^k}}{w}\mid w\in\cW\right\},\\
							{z^{k+1}} &=& \argmin_{z}\left\{L_{\sigma}(z,\bar w^k,\by^k;x^k)\mid z\in \Re^n\right\},\\
							w^{k+1} &=& \argmin_{w}\left\{L_{\sigma}(z^{k+1},w,\by^k;x^k) - \inprod{{\hat\delta_Q^k}}{w}\mid w\in\cW \right\},\\
							y^{k+1} &=& \argmin_{y}\left\{L_{\sigma}(z^{k+1},w^{k+1},y;x^k) - \inprod{\hat\delta_E^k}{y}\mid y\in \Re^m\right\},
						\end{array}
					\end{eqnarray*}
					where
					$\delta_E^k, \, \hat{\delta}_E^k \in \Re^{m}$, $\delta_Q^k,\,\hat\delta_Q^k \in\cW$ are error vectors such that
					\begin{eqnarray*}
						\max \{ \norm{\delta_E^k}, \norm{\hat \delta_E^k}, \norm{{\delta_Q^k}},
						\norm{{\hat\delta_Q^k}}\}\leq \epsilon_k.
					\end{eqnarray*}
					\item [Step 2.] Compute
					$x^{k+1} = x^k + \tau\sigma(z^{k+1} - Qw^{k+1}+A^*y^{k+1} -c).$				
				\end{description}
	}}}
	\caption{Algorithm QPPAL-Phase-I. \label{alg-qppal-phase-1}}
\end{figure}

The convergence of the Phase I algorithm follows from Theorem \ref{thmisPALM} and \ref{thmsGSisPALM} without much difficulty.
\begin{theorem}
	\label{thm:sGS-qsdp}
	Suppose that the solution set of {\rm({\bf P})} is nonempty and $A$ has full row rank. Let $\{(z^k,w^k,y^k,x^k)\}$ be the sequence generated by Algorithm {QPPAL-Phase-I}. Then, the sequence $\{(z^k,w^k,y^k)\}$ converges to an optimal solution of {\rm ({\bf D})} and $\{x^k\}$ converges to an optimal solution of {\rm ({\bf P})}.
\end{theorem}

{We note that there are two choices on the order of updating the dual variables $(z, w, y)$ in {\bf Step 1}. Specifically, one may choose either one of the two following orders: (1) $\bar{y}^k\rightarrow\bar{w}^k\rightarrow z^{k+1} \rightarrow w^{k+1} \rightarrow y^{k+1}$; (2) $\bar{w}^k\rightarrow\bar{y}^k\rightarrow z^{k+1} \rightarrow y^{k+1} \rightarrow w^{k+1}$. In fact, we observe that both updating rules have similar practical performance.}

\subsection{QPPAL Phase II}
In the second part of this section, we discuss our Phase II algorithm for solving  the convex quadratic programming ({\bf D}). The purpose of this phase is to obtain {highly} accurate solutions efficiently, with warm-starting by the Phase-I algorithm. As we shall see in the numerical experiments, the Phase II algorithm is indeed necessary and important for obtaining accurate solutions.

To proceed, we first note that problem ({\bf D}) has the following equivalent minimization form:
\begin{equation}\label{eq-D2}
- \min_{(w,y)\in  \mathcal{W}\times\Re^m} \left\{ h(w,y):=\delta_{\cC}^*(-Qw + A^*y - c) + \frac{1}{2}\left\langle w, Qw \right \rangle - \left\langle b,y \right \rangle \right\}.
\end{equation}
Then, we identify \eqref{eq-D2} with the problem of minimizing {$ h(w,y)=\tilde{h}(w,y,0) $ over $ \mathcal{W}\times \Re^m $, where}
\begin{equation*}
	\tilde{h}(w,y,\xi) = \delta_{\cC}^*(-Qw + A^*y - c + \xi) + \frac{1}{2}\left\langle w, Qw \right \rangle - \left\langle b,y \right \rangle,\quad {(w,y,\xi)\in \cW\times \Re^m \times \Re^n.}
\end{equation*}
Since $ \tilde{h} $ is jointly convex in $ (w,y,\xi) $, we are able to write down the Lagrangian function $ \tilde{l}: \mathcal{W} \times \Re^m \rightarrow \Re $ through partial dualization {(see for instance, \cite[Examples 11.46 and 11.57]{rockafellar2009variational})} as follows:
\begin{equation*}
	\tilde{l}(w,y;x):=\inf_{\xi} \left\{ \tilde{h}(w,y,\xi) - \left\langle x, \xi \right \rangle \right\} = \frac{1}{2}\left\langle w, Qw \right \rangle - \left\langle b,y \right \rangle - \left\langle x, Qw - A^*y+c \right \rangle - \delta_{\cC}(x).
\end{equation*}

Given $ \sigma>0 $, the augmented Lagrangian function corresponding to \eqref{eq-D2} in variables $ y,\,w $ and $ x $ can be obtained as follows:

\begin{equation*}
	\label{eqn-NewaugLag}
	\begin{aligned}
		\tilde{L}_\sigma(w,y;x)
		 := & \sup_{s\in \Re^n} \left\{ \tilde{l}(w,y;s) - \frac{1}{2\sigma} \norm{s - x}^2 \right\} \\
		 = &- \inf_{s\in \Re^n} \left\{ \left\langle s, Qw - A^*y+c \right \rangle + \delta_{\cC}(s) +  \frac{1}{2\sigma} \norm{s - x}^2\right\} +
		\frac{1}{2}\left\langle w, Qw \right \rangle  - \left\langle b,y \right \rangle \\
		 = & -\left\langle Qw - A^*y + c, \Pi_{\cC}[x - \sigma(Qw - A^*y + c)] \right \rangle  \\
		  & - \frac{1}{2\sigma} \norm{\Pi_{\cC}[x - \sigma(Qw - A^*y + c)] - x}^2 + \frac{1}{2}\left\langle w, Qw \right \rangle  - \left\langle b,y \right \rangle,
	\end{aligned}
\end{equation*}
{where $\Pi_{\cC}(\cdot)$ is the metric projector onto $C$, i.e., $\Pi_{\cC}(x):= \argmin_{x\in\Re^n}\{\norm{x-z}\mid z\in C\}$}.

We then propose to solve ({\bf D}) via an inexact proximal ALM. Its template is described in Fig.~\ref{alg-qppal-phase-2}.

\begin{figure}[ht!]
	\centerline{\fbox{\parbox{\textwidth}{
				{\bf Algorithm {QPPAL-Phase-II}: An inexact proximal ALM for solving ({\bf D})}
				\\
				Let $\sigma_0, \sigma_\infty >0$ be given parameters, and $ \{\tau_k\}_{k=0}^\infty $ be a given nonincreasing sequence such that $ \tau_k > 0 $ for all $ k\geq 0 $.
				Choose $(w^0,y^0)\in\cW\times\Re^m$ and $ x^0\in \mathbb{R}^n $.
				Set $k = 0$. Iterate the following steps.
				\begin{description}
					\item [Step 1.] Compute
					\begin{equation}\label{palm-sub}
						(w^{k+1},y^{k+1})
						\approx{}\textup{argmin}\left\{\begin{aligned}
							&\Psi_k(w,y): = \tilde{L}_{\sigma_k}(w,y; x^k) + \frac{\tau_k}{2\sigma_k}( \norm{w-w^k}^2_{Q}  \\[0pt]
							& + \norm{y - y^k}^2) \,\Big|\,
							w\in \cW,\;
							y\in\Re^{m}
						\end{aligned}\right\}.
					\end{equation}
					\item[Step 2.] Compute
					\[
						x^{k+1} = \Pi_{\cC} \left( x^k + \sigma_k(- Qw^{k+1} + A^*y^{k+1} - c)\right),\quad {z^{k+1} = \frac{1}{\sigma_k}\left(x^{k+1} - \left( x^k + \sigma_k(- Qw^{k+1} + A^*y^{k+1} - c)\right)\right)}.
					\]
					\item[Step 3.] Update $\sigma_{k+1} \uparrow \sigma_\infty\leq \infty$.
				\end{description}
	}}}
	\caption{Agorithm QPPAL-Phase-II. \label{alg-qppal-phase-2}}
\end{figure}

We next analyze the convergence of the algorithm {QPPAL-Phase-II} via establishing the connection between the proposed inexact proximal ALM and the preconditioned PPA studied in \cite{Li2019LP}, which extends the influential results in \cite{luque1984asymptotic,rockafellar1976augmented, rockafellar1976monotone}. To briefly explain the idea, let $ \mathcal{X}:= \cW \times \Re^m \times \Re^n $, and for $ k \geq 0$ and any given $ (\bar{w},\bar{y}, \bar{x}) $, define the function
\begin{equation}
	\label{eq-Pk}
	P_k(\bar{w},\bar{y}, \bar{x}):= \underset{(w,y,x)\in \mathcal{X}}{\mathrm{argminimax}} \left\{ {\hat{l}(w,y,x) :=} \tilde{l}(w,y,x) + \frac{\tau_k}{2\sigma_k} \left( \norm{y - \bar{y}}^2 + \norm{w - \bar{w}}^2_Q \right) - \frac{1}{2\sigma_k} \norm{x - \bar{x}}^2 \right\},
\end{equation}
{i.e., $P_k(\bar{w},\bar{y}, \bar{x})$ is obtained by minimizing and maximizing $\hat{l}(w,y,x)$ with respect to $(w,y)$ and $x$, respectively}. For the closed proper convex-concave function $ \tilde{l} $, define the maximal monotone operator $ \mathcal{T}_{\tilde{l}} $ by
\begin{eqnarray*}
	\mathcal{T}_{\tilde{l}}(w,y,x) & := & \left\{ (w',y',x')\,\Big|\, (w',y',-x')\in \partial \tilde{l}(w,y,x)  \right\} \\
	& = & \left\{ (w',y',x') \,\Big|\, w' = Q(w-x),\,y' = -b+Ax,\, x'\in Qw - A^*y + c + \partial \delta_{\cC}(x) \right\}.
\end{eqnarray*}
Notice that since $ \cC $ is polyhedral, $  \mathcal{T}_{\tilde{l}}  $ is a polyhedral set-valued mapping. Furthermore, since $ \mathcal{T}_{\tilde{l}} $ is a maximal monotone operator \cite{Minty1962}, its inverse exists and is given by
\begin{equation}
	\label{eq-invTl}
	\mathcal{T}_{\tilde{l}}^{-1}(w',y',x') := \underset{(w,y,x)\in \mathcal{X}}{\mathrm{argminimax}}\left\{ \tilde{l}(w,y,x) - \left\langle w',w \right \rangle - \left\langle y',y \right \rangle + \left\langle x',x \right \rangle \right\}.
\end{equation}
Then, the next lemma characterizes the optimal solution set in~\eqref{eq-Pk} whose proof can be found in Appendix \ref{subsec-lem-Pk}.
\begin{lemma}
	\label{lem-Pk}
	For all $ k\geq 0 $, let $\Lambda_k:= \mathrm{Diag}(\tau_k Q, \tau_k I_m,  I_n)$, which is positive definite {on} $ \mathcal{X} $, and for any $ (w,y,x)\in \mathcal{X} $, denote $\Lambda_k(w,y,x) = (\tau_kQw, \tau_k y, x)\in \mathcal{X} $. Then it holds that
	\begin{equation}
		\label{eq-lem-Pk}
		P_k(\bar{w},\bar{y},\bar{x}) = (\Lambda_k + \sigma_k \mathcal{T}_{\tilde{l}})^{-1} \Lambda_k(\bar{w},\bar{y},\bar{x}),\quad \forall\,(\bar{w},\bar{y},\bar{x})\in \cW \times \Re^m \times \Re^n .
	\end{equation}
	Moreover, $ P_k(w^*, y^*, x^*) = (w^*, y^*, x^*) $ if and only if $ (w^*, y^*, x^*) \in \mathcal{T}_{\tilde{l}}^{-1}(0) $.
\end{lemma}

Using Lemma \ref{lem-Pk}, the next proposition (see a proof in Appendix \ref{subsec-prop-inexactcondition}) allows us to propose a practical inexact rule (which implies the criteria used in \cite[Section 2]{Li2019LP}) for the inexact computation in~\eqref{palm-sub} via estimating the norm of the gradient of the function $\Psi_k(\cdot)$ {(see \eqref{palm-sub})} that is given by
\begin{equation*}
	\nabla \Psi_k(w,y) =
	\left[
	\begin{array}{c}
		Qw -  Q\Pi_{\cC}[x^k - \sigma_k(Qw - A^*y + c)] + \frac{\tau_k}{\sigma_k}Q(w - w^k)    \\
		-b + A\Pi_{\cC}[x^k - \sigma_k(Qw - A^*y + c)] + \frac{\tau_k}{\sigma_k}(y - y^k)
	\end{array}
	\right].
\end{equation*}

\begin{prop}
	\label{prop-inexactcondition}
	For any $ k = 0,1,\dots $, it holds that
	\begin{equation}
		\label{eq-prop-inexactcondition}
		\norm{(w^{k+1}, y^{k+1}, x^{k+1}) - P_k(w^k, y^k, x^k)}_{\Lambda_k} \leq \frac{\sigma_k}{\min\{1, \sqrt{\tau_k}, \sqrt{\tau_k\blue{\lambda_+(Q)}}\}}\norm{\nabla \Psi_k(w^{k+1}, y^{k+1})},
	\end{equation}
	\blue{where $\lambda_+(Q)>0$ denotes the smallest positive eigenvalue of $Q$.}
\end{prop}
Based on Proposition~\ref{prop-inexactcondition}, we then propose the following stopping criteria for the inexact computation in~\eqref{palm-sub}:
\begin{eqnarray*}
	&  {\bf (A)} & \quad \norm{\nabla \Psi_k(w^{k+1}, y^{k+1})} \leq  \frac{\min\{1, \sqrt{\tau_k}, \sqrt{\tau_k\blue{\lambda_+(Q)}}\}}{\sigma_k}\epsilon_k, \\
	&  {\bf (B)} & \quad \norm{\nabla \Psi_k(w^{k+1}, y^{k+1})} \leq \frac{\delta_k\min\{1, \sqrt{\tau_k}, \sqrt{\tau_k \blue{\lambda_+(Q)}}\}}{\sigma_k} \norm{(w^{k+1}, y^{k+1}, x^{k+1}) - (w^k, y^k, x^k)}_{\Lambda_k},
\end{eqnarray*}
where $\{\epsilon_k\}$ and $\{\delta_k\}$ are given nonnegative sequences such that $ \sum_{k=0}^\infty \epsilon_k < \infty $, and $\delta_k< 1,\; \sum_{k=0}^\infty \delta_k < \infty $. Thus, we can directly present the convergence properties of the proposed algorithm in the following theorem which combines the results in \cite[Theorem 1 \& Theorem 2]{Li2019LP} by observing that Algorithm QPPAL-Phase-II actually computes $(w^{k+1}, y^{k+1}, x^{k+1})\approx P_k(w^k, y^k, x^k) = (\Lambda_k + \sigma_k \mathcal{T}_{\tilde{l}})^{-1}(w^k,y^k,x^k)$. We omit the proof here since it can be done exactly the same way as in \cite{Li2019LP}.
\begin{theorem}\label{thm-convergece-alm}
	Suppose that the solution set of {\rm({\bf P})} and {\rm({\bf D})} is nonempty, $A$ has full row rank, and the positive sequence $ \{\tau_k \} $ is non-increasing and bounded away from zero, i.e., $ \tau_k\downarrow\tau_\infty>0 $. Let $ \{(w^k, y^k, x^k)\} $ be the sequence generated by Algorithm QPPAL-Phase-II.
	\begin{enumerate}
		\item If the algorithm is executed under the {inexactness} condition $ {\bf (A)} $, then the sequence $ \{(w^k, y^k, x^k)\} $ is bounded. Furthermore, $ \{x_k\} $ converges to an optimal solution of {\rm({\bf P})} and $ \{(w^k,y^k)\} $ converges to an optimal solution of {\rm({\bf D})}.
		\item Let $ r > \sum_{k=0}^{\infty} \epsilon_k $ be any positive constant and $\kappa >0$ be the corresponding error bound constant\footnote{The existence of such $r$ and $\kappa$ associated with the polyhedral multifunction $\mathcal{T}_{\tilde l}$ \cite{sun1993convergence} can be derived from the classic error bound result in \cite{Robison1981some}.	See, for example, \cite[Lemma 2.4]{Li2019LP}.} such that
		\begin{equation*}
			\mathrm{dist}((w,y,x), \mathcal{T}_{\tilde{l}}^{-1}(0)) \leq \kappa \mathrm{dist}(0,\mathcal{T}_{\tilde{l}}(w,y,x)), \quad \forall\, (w,y,z) \textrm{ s.t. } \mathrm{dist}((w,y,x), \mathcal{T}_{\tilde{l}}^{-1}(0)) \leq r.
		\end{equation*}
		Moreover, suppose that the initial point $ (w^0, y^0, x^0) $ satisfies $ \mathrm{dist}((w^0,y^0,x^0), \mathcal{T}_{\tilde{l}}^{-1}(0)) \leq r - \sum_{k=0}^{\infty}\epsilon_k $ and the proposed algorithm is executed under both conditions $ {\bf (A)} $ and $ {\bf (B)} $. Then, for all $ k\geq 0 $, it holds that
		\begin{equation}
			\label{eq-thm-convegence}
			\mathrm{dist}_{\Lambda_k}\left( (w^{k+1},y^{k+1},x^{k+1} ), \mathcal{T}_{\tilde{l}}^{-1}(0) \right) \leq \mu_k \mathrm{dist}_{\Lambda_k}\left( (w^{k},y^{k},x^{k} ), \mathcal{T}_{\tilde{l}}^{-1}(0) \right),
		\end{equation}
		where $ \mu_k = (1-\delta_k)^{-1} \left( \delta_k + (1+ \delta_k)\kappa \gamma_k / \sqrt{\sigma_k^2 + \kappa^2 \gamma_k^2} \right) $ with $ \gamma_k:=\max\{1, \tau_k, \tau_k\norm{Q}_2\} $ and
		\begin{equation*}
			\limsup_{k\rightarrow \infty} \mu_k = \mu_\infty = \frac{\kappa\gamma_\infty}{\sqrt{\sigma_{\infty}^2 + \kappa^2 \gamma_\infty^2}} < 1, \quad (\mu_\infty := 0 \textrm{ if } \sigma_\infty = \infty){,}
		\end{equation*}
		with $ \gamma_\infty = \max\{1, \tau_\infty, \tau_\infty \norm{Q}_2\} $.
	\end{enumerate}
\end{theorem}
Since \blue{$ 0 < \inf_k \min\{\tau_k\lambda_+(Q), \tau_k, 1\} \leq \sup_k \max\{\tau_k\norm{Q}_2, \tau_k, 1\} < \infty $}, the distance induced by $ \Lambda_k $ can be replaced by the Euclidean distance without much difficulty. Moreover, the above theorem shows that the linear rate $\mu_k$ can be arbitrarily small if $\sigma_k$ is sufficiently large, i.e., the linear convergence of the algorithm can be ``arbitrarily fast''. However, in practice, it is not advisable to choose $\sigma_k$ to be extremely large for the purpose of numerical stability. Therefore, given that $\sigma_k\leq \sigma_\infty < \infty$, a smaller $\tau_k$ will lead to a better linear convergence rate, i.e., a smaller $\mu_k$. So for better theoretical performance, one prefers to choose a smaller $\tau_k$. In fact, Theorem \ref{thm-convergece-alm} indicates that ideally we would choose $\tau_k\leq \min\{1, \norm{Q}_2^{-1}\}$ for better convergence rate.

To summarize, we shall present our {conceptual} two-phase algorithm QPPAL in Fig.~\ref{alg-qppal}. {Readers may refer to Section \ref{sec-numerical} for the termination rules with respect to a given tolerance of our algorithms. Note that ${\tt tol}_1$ and ${\tt MaxIter}_1$ may  be {chosen differently for different problems to achieve greater}
efficiency.}
\begin{figure}[ht!]
	\centerline{\fbox{\parbox{\textwidth}{
				{\bf Algorithm {QPPAL}}: {\bf A two-phase algorithm for ({\bf D}).}
				{
				\begin{description}
					\item [Input.] Stopping tolerances ${\tt tol}_1$ and ${\tt tol}_2$, and maximum number of iterations ${\tt MaxIter}_1$ and ${\tt Maxiter}_2$.
					\item [Step 1.] Run Algorithm {QPPAL-Phase-I} and terminate it when either one of the following two conditions is met: \\
						(1) Optimality and feasibility measure is less than ${\tt tol}_1$; (2) The number of iterations reaches ${\tt MaxIter}_1$.
					\item [Step 2.] Run Algorithm {QPPAL-Phase-II} with the initial point given in {\bf Step 1}, and terminate it when either one of the following two conditions is met:\\
						(1) Optimality and feasibility measure is less than ${\tt tol}_2$; (2) The number of iterations reaches ${\tt MaxIter}_2$.
				\end{description}
				}
	}}}
	\caption{Algorithm QPPAL. \label{alg-qppal}}
\end{figure}

\subsection{A semismooth Newton method for solving \texorpdfstring{\eqref{palm-sub}}{palm}}

In this subsection, we discuss how to solve the subproblem in~\eqref{palm-sub} efficiently. To this end, for given $ (\hat{w}, \hat{y}, \hat{x})\in \cW\times \Re^m\times \Re^n $, $ \tau > 0 $ and $ \sigma >0 $, we define the function
\begin{equation*}
	\varphi(w,y):= \tilde{L}_\sigma(w,y;\hat{x}) + \frac{\tau}{2\sigma}\left( \norm{w - \hat{w}}_Q^2 + \norm{y - \hat{y}}^2 \right), \quad  \forall\,(w,y)\in \cW\times \Re^m,
\end{equation*}
whose gradient is given by
\begin{equation*}
	\nabla \varphi(w,y) =
	\left[
	\begin{array}{c}
		Qw -  Q\Pi_{\cC}(z(w,y)) + \frac{\tau}{\sigma}Q(w - \hat{w})    \\
		-b + A\Pi_{\cC}(z(w,y)) + \frac{\tau}{\sigma}(y - \hat{y})
	\end{array}
	\right], \quad (w,y)\in \cW\times \Re^m,
\end{equation*}
where $ z(w,y):=\hat{x} - \sigma(Qw - A^*y + c) $. Note that solving the minimization problem
\begin{equation}
	\label{eq-palm-sub-2}
	\min_{{w,y}}\left\{ \varphi(w,y)~\Big|~ (w,y)\in \cW\times \Re^m \right\}
\end{equation}
is equivalent to solving the following system of nonlinear equations:
\begin{eqnarray}\label{eq-wy-nonsmooth}
	\nabla\varphi(w,y) = 0, \quad (w,y)\in\cW\times\Re^m.
\end{eqnarray}

Since $\cC$ is a polyhedral set,  $\Pi_{\cC}(\cdot)$ is piecewise linear and hence strongly semismooth. Thus, we can design a semismooth Newton (SSN) method to solve \eqref{eq-wy-nonsmooth} and could expect  a superlinear or even quadratic convergence rate. For any $(w,y)\in\cW\times\Re^m$, define
\[\hat\partial^2 \varphi(w,y) := \left[ \begin{array}{cc}
	Q &   \\
	& 0
\end{array} \right]+\sigma \left[ \begin{array}{c}
	Q  \\
	-A
\end{array} \right]\partial\Pi_{\cC}(z(w,y)){\left[ \begin{array}{c}
	Q  \\
	-A
\end{array} \right]^*} +
\frac{\tau}{\sigma}\left[\begin{array}{cc}
	Q & \\
	& I \end{array}\right], \]
where \blue{$\partial\Pi_{\cC}(z(w,y))$} is the Clarke subdifferential \cite{Clarke83} of $\Pi_{\cC}(\cdot)$ at $z(w,y)$. Note that from \cite{hiriart1984generalized}, we know that
\begin{eqnarray}
	\hat{\partial}^2 \varphi(w,y)\, (d_w;d_y) = {\partial}^2 \varphi(w,y)\, (d_w;d_y), \quad  \forall \, (d_w;d_y) \in \cW\times\Re^m,
	\label{eq-Clarke}
\end{eqnarray}
where ${\partial}^2 \varphi(w,y)$ denotes the generalized Hessian of $\varphi$ at $(w,y)$, i.e., the Clarke generalized Jacobian of $\nabla \varphi$ at $(w,y)$. Given $(w,y)\in\cW\times\Re^m$, let $U \in\partial\Pi_{\cC}(z(w,y))$ and
\begin{equation}\label{p2:eq-netwon-partial}
	V = \left[ \begin{array}{cc}
		Q &   \\
		& 0
	\end{array} \right]+\sigma \left[ \begin{array}{c}
		Q  \\
		-A
	\end{array} \right]U{\left[ \begin{array}{c}
	Q  \\
	-A
\end{array} \right]^*} +
	\frac{\tau}{\sigma}\left[\begin{array}{cc}
		Q & \\
		& I \end{array}\right] .
\end{equation}
Then, we have $V\in\hat\partial^2\varphi(w,y)$.

After all the preparations, we can design a semismooth Newton method (see Fig.~\ref{alg-SSN}) as in \cite{SDPNAL} to solve \eqref{eq-wy-nonsmooth}.

\begin{figure}[ht!]
	\centerline{\fbox{\parbox{\textwidth}{
				{\bf Algorithm {SSN}}: {\bf A semismooth Newton algorithm.}
				\\
				Given {$\bar{\eta} \in (0, 1)$}, $\nu \in (0,1]$, $\delta \in (0, 1)$ and ${\rho} \in (0, 1/2)$. Choose $(w^0,y^0)\in\cW\times\Re^m$. {Iterate the following steps for $j = 0,1,\dots,$.}
				\begin{description}
					\item[Step 1.]  Find an approximate solution $(d_w^j;d_y^j)\in\cW\times\Re^m$ to
					\begin{eqnarray}\label{eqn-epsk}
						V_j(d_w;d_y) = -\nabla \varphi(w^j,y^j)
					\end{eqnarray}
					such that
					\[
					\norm{V_j(d_w^j;d_y^j) + \nabla \varphi(w^j,y^j)}\le \eta_j := \min(\bar{\eta}, \| \nabla \varphi(w^j,y^j)\|^{1+\nu}),
					\]
					where $V_j \in \hat{\partial}^2 \varphi(w^j, y^j)$ is defined as in \eqref{p2:eq-netwon-partial} with $U_j\in \partial \Pi_{\cC}\big(z(w^j, y^j)\big)$.
					\item[Step 2.]  Set $\alpha_j = \delta^{m_j}$, where $m_j$ is the first nonnegative integer $m$ for which
					\begin{eqnarray}\label{Armijo}
						\varphi(w^j + \delta^{m} d_w^j,y^j+\delta^m d_y^j) \leq \varphi(w^j,y^j) + {\rho} \delta^{m}
						\langle \nabla \varphi(w^j,y^j), (d^j_w;d^j_y) \rangle.
					\end{eqnarray}
					\item[Step 3.] Set $w^{j+1} = w^j + \alpha_j \, d_w^j$ and $y^{j+1} = y^j + \alpha_j \, d_y^j$.
				\end{description}
	}}}
	\caption{Algorithm SSN. \label{alg-SSN}}
\end{figure}

The convergence results for the above {SSN} algorithm are stated in Theorem \ref{convergence-zwy-newton}.
\begin{theorem}\label{convergence-zwy-newton}
	Let the sequence $\{(w^j,y^j)\}$ be generated by Algorithm {SSN}. Suppose at each step $j\ge 0$, the tolerance $\eta_j$ is achieved, i.e.,
	\begin{equation*}
		\norm{V_j(d_w;d_y) + \nabla \varphi(w^j,y^j)}\le \eta_j.
	\end{equation*}
	Then the  sequence $\{(w^j,y^j)\}$ converges to the unique optimal solution, say $(\bar w, \bar y)$, of the optimization problem in \eqref{eq-palm-sub-2} and
	\begin{equation}
		\|(w^{j+1},y^{j+1}) - (\bar w,\bar y) \| = O(\norm{(w^j,y^j) - (\bar w,\bar y)})^{1+\nu}.
	\end{equation}
\end{theorem}
\begin{proof}
	Note that  $\Pi_{\cC}(\cdot)$ is strongly semismooth. Since $\varphi(w,y)$ is a strongly convex function defined on $\cW\times\Re^n$, problem \eqref{eq-palm-sub-2} then has a unique solution $(\bar w,\bar y)$ and the level set $\{(w,y)\in\cW\times\Re^m \,|\, \varphi(w,y)\le \varphi(w^0,y^0)\}$ is compact. Therefore, the sequence generated by {SSN} is bounded as $(d_w^j,d_y^j)$ is a descent direction \cite[Propsition 3.3]{SDPNAL}. Note that for all $(w,y)\in\cW\times\Re^n$, every $V \in\hat\partial^2 \varphi(w,y)$ is self-adjoint and positive definite on $ \cW\times\Re^n$. Thus, the desired convergence can be easily obtained by combining \cite[Theorem 3.4 \& 3.5]{SDPNAL}.
\end{proof}

In Theorem \ref{convergence-zwy-newton}, it is clear that $V_j$ in the Newton system \eqref{eqn-epsk} in the form of \eqref{p2:eq-netwon-partial} is guaranteed to be positive definite as a positive definite proximal term is added. Indeed, adding the proximal term in our algorithmic design relieves the need of requiring additional conditions, such as the constraint nondegenerate condition (see e.g., \cite{SDPNAL}), to ensure the nonsingularity of  $V_j$ in \eqref{eqn-epsk}. Moreover, to improve the condition number of the corresponding coefficient matrix, we would prefer a larger $\tau_k$. However, to obtain better convergence rate for Algorithm QPPAL-Phase-II, we want a smaller $\tau_k$. The two opposing effects imply that in the implementation of the algorithm we need to choose the parameter $\tau_k$ appropriately to balance the efficiency and robustness of the proposed algorithm.

\section{Implementation of QPPAL}
\label{sec-implementation}
{In this section, we provide implementation details for our proposed QPPAL. Note that in both phase one and phase two algorithms, the main difficulty is to solve several types of {linear systems}. Thus, our main focus in this section is on how to solve such linear systems efficiently. Moreover, since choosing the parameters $\tau_k$ and $\sigma_k$ {plays} an important role in the numerical implementation of the QPPAL-Phase-II algorithm, we also
describe at the end of this section on how to {adjust} these parameters dynamically based on the optimality and infeasibility measures that will be defined later.}

\subsection{Implmentation of QPPAL-Phase-I}
\label{subsec-implementation-phase-I}
{We first} discuss how to perform {\bf Step 1} in Algorithm QPPAL-Phase-I efficiently. Firstly, in order to obtain $ \bar{y}^k $, a system of linear equations of the following form is solved:
\begin{equation*}
	-(b + \delta_E^k) + \sigma A(z^k - Qw^k + A^*y - c+ \sigma^{-1}x^k ) = 0.
\end{equation*}
{
By simple calculations, we derive that
\begin{equation*}
	 {(\sigma AA^*)\, \bar{y}^k = g^k +  \delta_E^k,}
\end{equation*}
where $g^k = (b - Ax^k)- \sigma A(z^k -Qw^k - c)$.
Similarly, $ y^{k+1} $ is computed as follows:
\begin{equation*}
	{(\sigma AA^*)\, y^{k+1} = g^{k+1}
	+  \hat{\delta}_E^k,}
\end{equation*}
where $g^{k+1} = (b - Ax^k )- \sigma A(z^{k+1} -Qw^{k+1} - c) $.
Note that both $ \delta_E^k $ and $ \hat{\delta}_E^k $ should be
interpreted as the residual vectors obtained by  solving
the above linear systems without these terms on the right-hand-sides.
Observe that when the sparse Cholesky factorization of $AA^*$ (which only need
to be computed once) can be obtained at a moderate cost,
we can solve these two linear systems by backward-forward substitutions cheaply
in each iteration of QPPAL-Phase-I.
In the event when solving the linear systems directly is too costly or not possible due to the memory
constraint,
one may use an iterative solver, such as
a preconditioned minimum residual method (MINRES) \cite{Paige-Saunders1975} or
a preconditioned symmetric quasi-minimal residual method (PSQMR) \cite{Freund1994}, to solve the target linear systems. In the
latter situation, the
residual vectors just need to satisfy the conditions that
$$
 \norm{\delta_E^k} \leq \epsilon_k, \quad  \norm{\hat{\delta}_E^k} \leq \epsilon_k,
$$
where $\{\epsilon_k\}$ is a given nonnegative summable sequence, for example, one
may choose $\epsilon_k = \min\{ 10^{-2},1/k^{1.1}\}$ for all $k\geq 1.$
Observe that the accuracy requirement on the residual norm is actually quite mild,
since for the previous choice, $\epsilon_k \approx 4\times 10^{-5}$ even when
$k$ is as large as $10000$.

In order to obtain $\bar w^k$ and $w^{k+1}$, we need to solve a system of linear equations
of the form
\begin{equation}
	\label{eq-qppal-w}
	(Q +\sigma Q^2)  w \approx Qh,\quad w\in\cW,
\end{equation}
where $h\in\Re^n$ is a given vector.
In particular, when computing $\bar{w}^k$,
$h:= \sigma (z^k - A^*\bar{y}^k - c) + x^k $, and when computing $w^{k+1}$, $h:= \sigma (z^{k+1} - A^*\bar{y}^k - c) + x^k $.
Typically, we can only afford to solve the linear system \eqref{eq-qppal-w} by an iterative solver
inexactly. In this case, we require
the residual vector $\delta_Q$ corresponding to a computed
solution $\hat{w}$ to satisfy the condition that
\begin{equation}
	\label{eq-qppal-w-r}
	\norm{\delta_Q} = \norm{Qh- (Q+ \sigma Q^2) \hat{w}}\le \epsilon_k.
\end{equation}
}

Note that there is a unique solution which solves \eqref{eq-qppal-w} exactly {since $Q$ is positive definite on the subspace $\cW$}. Under the high dimensional setting where $n$ is huge and the matrix representation of $Q$ may not be available, \eqref{eq-qppal-w} can only be solved inexactly by an iterative method. Indeed, based on our numerical experiments for solving QP relaxations for certain classes of integer programming problems and the QP problems arising from portfolio optimization, matrices $Q$ in these problems are usually fully dense and large-scale. Hence, a direct solver may not be sufficiently efficient. Moreover, due to the presence of the subspace constraint $w\in\cW$, it is apparently difficult to solve \eqref{eq-qppal-w} if $\cW\neq \Re^n$. Fortunately, we are able to propose the following strategy to rectify this difficulty. Instead of solving \eqref{eq-qppal-w}, we propose to solve the following simpler linear system
\begin{equation}\label{eq-qppal-w-2}
	(I + \sigma Q) w \approx h,
\end{equation}
with the {residual norm of a computed solution $\hat{w}$ satisfying}
\begin{equation}\label{eq-qppal-w-r2}
	\norm{h-(I+\sigma Q) \hat{w} }\le \frac{\epsilon_k}{\norm{Q}_2}.
\end{equation}
{The following simple spectral analysis will show that the coefficient matrix in \eqref{eq-qppal-w-2} generally would have a much better conditioning than that of \eqref{eq-qppal-w}. Suppose $\lambda_{1}$, $\lambda_{n}$, $\lambda_{n}^+$ denote the largest, smallest, and smallest positive, eigenvalue of $Q$, respectively. Based on standard convergence analysis of an iterative method such as MINRES, we know that its convergence rate for solving \eqref{eq-qppal-w} is determined by the condition number $\kappa_1 = \frac{\lambda_1}{\lambda_n^+}\frac{1+\sig\lambda_1}{1+\sig\lambda_n^+}$, whereas the convergence rate for solving \eqref{eq-qppal-w-2} is determined by $\kappa_2 = \frac{1+\sig\lambda_1}{1+\sig\lambda_n}$. For the case that $\lambda_n^+ = \lambda_n$, i.e., $Q$ is positive definite, it is clear that $\kappa_2 = \frac{\lambda_n}{\lambda_1}\kappa_1 \leq \kappa_1$. For the case where $\lambda_n^+ > \lambda_n = 0$, i.e., $Q$ is singular, it is clear that $\kappa_2$ would be smaller than $\kappa_1$ if $\lambda_n^+(1+\sig\lambda_n^+) < \lambda_1$, or equivalently $\lambda_n^+ < \frac{1}{2}(\sqrt{\sig^2+4\lambda_1}-\sig)$. One would expect the latter condition to hold for most matrices $Q$ unless $\lambda_n^+$ is large.
}

We can apply an iterative method (e.g., PSQMR, MINRES) to solve \eqref{eq-qppal-w-2} to obtain an approximate solution $\hat w$ such that \eqref{eq-qppal-w-r2} holds for $\hat w$. Then
\[
\norm{\delta_Q} =
\norm{Q(h - (I+\sigma Q)\hat w)} \le \norm{Q}_2\,\norm{h-(I+\sigma Q)\hat w } \le \norm{Q}_2\frac{\epsilon_k}{\norm{Q}_2} = \epsilon_k.
\]
Thus, we have that $w^* = \Pi_{\cW}(\hat w)\in\cW$ solves \eqref{eq-qppal-w} with the corresponding residual satisfying \eqref{eq-qppal-w-r}. Surprisingly, much to our delight, it is not necessary for us to compute {$w^*$} explicitly since to update the iterations in Algorithm {QPPAL-Phase-I}, we only need to compute $Qw^*$ which is easily shown to be equal to $Q\hat{w}$.  Hence, we only need to solve the linear system \eqref{eq-qppal-w-2} to obtain a approximate solution $\hat w$ and then compute $Q\hat w$.

\subsection{Implementation of QPPAL-Phase-II}
\label{subsec-implementaion-phase-II}
{We next illustrate} how to solve the linear system~\eqref{eqn-epsk} at each iteration of \textrm{SSN} efficiently. Notice that for a given $ \sigma >0 $, $ \hat{x}\in \Re^n $ and $ z(w,y) = \hat{x} - \sigma (Qw - A^*y + c) $, we can choose $ U\in \partial \Pi_{\cC}(z(w,y)) $ to be a diagonal matrix of order $ n $ whose diagonal entries are given as follows:
\begin{equation*}
	U_{ii} =
	\left\{
	\begin{array}{ll}
		1 & l_i < (z(w,y))_i < u_i \\
		0 & \textrm{otherwise}
	\end{array}
	\right., \quad 1\leq i\leq n.
\end{equation*}
Thus, the $(n+m)\times (n+m)$ coefficient matrix is given by
\begin{equation}
	\label{eq-matV}
	V = \left[
	\begin{array}{cc}
		(1+\frac{\tau}{\sigma}) Q + \sigma QUQ & -\sigma QUA^* \\
		-\sigma AUQ & \frac{\tau}{\sigma}I_m + \sigma AUA^*
	\end{array}
	\right],
\end{equation}
for a given $ \tau >0 $. Recall that $ V $ is positive definite on $ \cW\times \Re^m $ and hence for any $ {r}\in \cW\times \Re^m  $, the linear system
\begin{equation}
	\label{eq-ssn-linsys}
	V(d_w; d_y) = {r},\quad (d_w;d_y)\in \cW\times \Re^m{,}
\end{equation}
has a unique solution. In the following discussion, we always take
\begin{equation*}
	{r} :=
	\left[
	\begin{array}{c}
		Q{r_1} \\
		{r_2}
	\end{array}
	\right]:=-
	\left[
	\begin{array}{c}
		Q\left( w - \Pi_{\cC}(z(w,y)) + \frac{\tau}{\sigma}(w - \hat{w}) \right) \\
		-b + A \Pi_{\cC}(z(w,y)) + \frac{\tau}{\sigma} (y - \hat{y})
	\end{array}
	\right] = -\nabla \varphi(w,y).
\end{equation*}
Since $Q$ is possibly a large dimensional and dense matrix, applying a direct method to solve \eqref{eq-ssn-linsys} may not be practical. Moreover, matrix-vector multiplications involving $Q$ could be expensive. Therefore, iterative solvers such as PSQMR for solving \eqref{eq-ssn-linsys} may also be expensive. To resolve this issue, instead of solving \eqref{eq-ssn-linsys} directly, we solve a simpler linear system to compute $ Qd_w $ approximately via solving a nonsymmetric linear system. In particular, we shall use the BICGSTAB method studied in \cite{Saad2003} to solve the new system. The next proposition (see \cite[Proposition 4.1]{Li2018QSDPNAL}) demonstrates this approach and further implies that only one matrix-vector multiplication with respect to $ Q $ is required in each BICGSTAB iteration. This indeed reduces the computational cost compared with using $ V $ directly (especially when $Q$ is dense), since the latter requires two such matrix-vector multiplications in each PSQMR iteration.

\begin{prop}
	Let the matrix $ V $ be given by~\eqref{eq-matV}, and denote
	\begin{equation}
		\label{eq-prop-Vhat}
		\hat{V}:=
		\left[
		\begin{array}{cc}
			(1+\frac{\tau}{\sigma})I_n + \sigma UQ & -\sigma UA^* \\
			-\sigma AUQ & \frac{\tau}{\sigma}I_m + \sigma AUA^*
		\end{array}
		\right].
	\end{equation}
	Suppose $ (\hat{d}_w; \hat{d}_y) $ is an approximate solution to the following system:
	\begin{equation}\label{eq-V-assymetric}
		\hat{V}(\hat{d}_w; \hat{d}_y) \approx ({r_1}; {r_2})
	\end{equation}
	with the residual satisfying
	\begin{equation*}
		\norm{{\hat{V}}(\hat{d}_w; \hat{d}_y) - ({r_1}; {r_2})} \leq \frac{\epsilon}{\max\{\norm{Q}_2, 1\}}.
	\end{equation*}
	Let $ d_w:=\Pi_{\cW}(\hat{d}_w)\in \cW $.
	Then $ (d_w, \hat{d}_y) \in \cW\times \Re^m $ solves~\eqref{eq-ssn-linsys} with the residual norm satisfying
	\begin{equation*}
		\norm{V(d_w; \hat{d}_y) - (Q{r_1}; {r_2})}\leq \epsilon.
	\end{equation*}
	Moreover,
	\begin{equation*}
		Qd_w = Q\hat{d}_w, \quad \left\langle d_w, Qd_w \right \rangle =
		\left\langle \hat{d}_w, Q\hat{d}_w \right \rangle.
	\end{equation*}
\end{prop}

Again, similar to the case in Algorithm QPPAL-Phase-I, when updating the variable $w$, we do not need to compute $d_w$ explicitly since we can safely excute the algorithm by only updating $Qw$, namely, computing $Qd_w$. The fact that one can replace~\eqref{eq-ssn-linsys} by the simpler linear system \eqref{eq-V-assymetric} is a powerful feature of our proposed algorithm.

Finally, we can further reduce the size of the linear system in \eqref{eq-prop-Vhat} by exploiting the special structure of the diagonal matrix $U$. To this end, we assume without loss of generality that $U$ has the following representation
\begin{equation*}
	U =
	\left[
	\begin{array}{cc}
		I_p & 0 \\
		0 & 0
	\end{array}
	\right]\in \Re^{n\times n},\quad I_p\in \Re^{p\times p},\; 0\leq p \leq n.
\end{equation*}
Based on the above representation, we can then partition the vectors $\hat{d}_w$ {and
$r_1$}, the matrices $A$ and $Q$ accordingly as follows:
\begin{equation*}
	\hat{d}_w = \left[
	\begin{array}{c}
		\hat{d}_w^P \\  \hat{d}_w^Z
	\end{array}
	\right],\quad {r_1} =
	\left[
	\begin{array}{c}
		{r_1}^P \\  {r_1}^Z
	\end{array}
	\right],\quad
	A =
	\left[
	\begin{array}{cc}
		A_P & A_Z
	\end{array}
	\right],\quad Q =
	\left[
	\begin{array}{cc}
		Q_{PP} & Q_{PZ} \\
		Q_{PZ}^T & Q_{ZZ}
	\end{array}
	\right],
\end{equation*}
where $\hat{d}_w^P\in \Re^{p}$, $\hat{d}_w^Z \in \Re^{n-p}$, ${r_1}^P\in \Re^{p}$, ${r_1}^Z \in \Re^{n-p}$, $A_P\in \Re^{m\times p},\; A_Z\in \Re^{m\times (n-p)}$, $Q_{PP}\in \Re^{p\times p}$, $Q_{PZ}\in \Re^{p\times (n-p)}$ and $Q_{ZZ}\in \Re^{(n-p)\times (n-p)}$. Moreover, simple calculations show that
\[
UQ =
\left[
\begin{array}{cc}
	Q_{PP} & Q_{PZ} \\
	0 & 0
\end{array}
\right],\quad
UA^* =
\left[
\begin{array}{c}
	A_P^* \\  0
\end{array}
\right],\quad
AUA^* =  A_PA_P^* .
\]	
For notational simplicity, we denote $\nu := \sigma^{-1}\tau$. Based on	the aforementioned partitions, we rewrite the linear system \eqref{eq-V-assymetric} as follows:
\[
\begin{aligned}
	(1+\nu)\hat{d}_w^Z &= {r_1}^Z, \\
	\left((1+\nu)I_p + \sigma Q_{PP}\right)\hat{d}_w^P - \sigma A_P^*\hat{d}_y & = {r_1}^P - \sigma Q_{PZ}\hat{d}_w^Z = {r_1}^P - \sigma(1+\nu)^{-1}Q_{PZ}{r_1}^Z =: {\bar{r}_1}, \\
	-\sigma A_PQ_{PP}\hat{d}_w^P + \left(\nu I_m + \sigma A_P A_P^*\right)\hat{d}_y & = {r_2} + \sigma A_P Q_{PZ} \hat{d}_w^Z = {r_2} + \sigma (1+\nu)^{-1}A_PQ_{PZ}{r_1}^Z =: {\bar{r}_2}.
\end{aligned}
\]
Now, by writing the third equation as
\[
- A_P\big((1+\nu)I_p +\sigma Q_{PP}\big) \hat{d}_w^P  + (1+\nu)A_P \hat{d}_w^P
+ \left(\nu I_m + \sigma A_P A_P^*\right)\hat{d}_y
= {\bar{r}_2}
\]
and making use of the second equation, we get after some simple manipulations that
\begin{eqnarray}
	\hat{d}_y = \nu^{-1} \Big( A_P{\bar{r}_1} +{\bar{r}_2}  - (1+\nu)A_P\hat{d}_w^P\Big) .
	\label{eq-hatdy}
\end{eqnarray}

By using the above expression of $\hat{d}_y$ in  the second equation, we get
\begin{eqnarray}
	\left((1+\nu)I_p + \sigma Q_{PP} + \nu^{-1}(1+\nu)\sigma  A_P^*A_P\right)\hat{d}_w^P
	& =& {\bar{r}_1 + \nu^{-1}\sigma} A_P^*\left(A_P{\bar{r}_1} + {\bar{r}_2}\right).
	\label{eq-hatdwP}
\end{eqnarray}
It is obvious that the new target linear system \eqref{eq-hatdwP} has a symmetric positive definite coefficient matrix of size $p \leq n$.	Therefore, we can apply a direct solver to solve \eqref{eq-hatdwP} via computing the Cholesky factorization of the coefficient matrix when $p \ll n$ or an iterative solver such as PSQMR when $ p \approx n$. Observe that by exploiting the active-set structure in $U$, we only need to  solve a smaller-scale problem of dimension $p\times p$ instead of the $(n+m)$-dimensional problem \eqref{eq-V-assymetric}.

As a conclusion, instead of solving the non-symmetric linear system \eqref{eq-V-assymetric}, we can solve the  smaller symmetric  positive definite linear system \eqref{eq-hatdwP} for $\hat{d}_w^P$. Once that is computed, we can obtain $\hat{d}_y$ from \eqref{eq-hatdy}. We should mention that while \eqref{eq-hatdwP} appears to be more appealing than \eqref{eq-V-assymetric}, the former can be much more ill-conditioned than the latter when $\sigma$ is large. Thus when \eqref{eq-hatdwP} itself is large-scale and requires an iterative solver, it would be more efficient to  apply the BICGSTAB solver to \eqref{eq-V-assymetric} directly when $\sigma$ is large.

\subsection{Updating rules for \texorpdfstring{$\tau_k$ and $\sigma_k$}{tausig}}
\label{subsec-tau-sigma}
{Before describing the updating rule for $\tau_k$ and $\sigma_k$, we shall define the optimality and infeasibility measures of an approximate optimal solution $(x^k,z^k,w^k,y^k)\in \Re^n\times\Re^n\times \cW\times \Re^m$ for problems ({\bf P}) and ({\bf D}) at the $k$-th iteration of QPPAL-Phase-II algorithm. In particular, we define the following relative KKT residuals:
\[
    \eta_p^k :=  \frac{\norm{Ax^k - b}}{1 + \norm{b}},\quad \eta_d^k :=  \frac{\norm{z^k - Qw^k + A^*y^k-c}}{1 + \norm{c}}, \quad \eta_Q^k :=  \frac{\norm{Qx^k - Qw^k}}{1 + \norm{Qx^k} + \norm{Qw^k}},\quad \eta_{\cC}^k:=\frac{\norm{x^k - \Pi_{\cC}(x^k - z^k)}}{1 + \norm{x^k} +\norm{z^k}}.
\]
Then, starting from $\sigma_0 > 0$ and $\tau_0 > 0$, we update $\sigma_k$ as follows:
\[
    \sigma_{k} = \left\{
        \begin{array}{ll}
        \frac{5}{4}\sigma_{k-1}  &  \textrm{if } \; \max\{\eta_p^k,\eta_Q^k,\eta_{\cC}^k\} < {\frac{3}{4}\eta_d^k},\\
            \frac{4}{5}\sigma_{k-1}  &  \textrm{if } \; {\eta_d^k} < \frac{3}{4}\max\{\eta_p^k,\eta_Q^k,\eta_{\cC}^k\},\\
            \sigma_{k-1} & \textrm{otherwise,}
        \end{array}
    \right.\quad \forall \; k \geq 1.
\]
Moreover, we update $\tau_k$ such that $\frac{\tau_k}{\sigma_k} = \max\{10^{-12}, \kappa\cdot k^{-2.5}\}$, for $k\geq 1$, where $\kappa>0$ is a given constant that depends on the problems to be solved. Typically, $\kappa$ is chosen to be in the
range $[10^{-4}, 10^{2}].$
}

\section{Numerical experiments}
\label{sec-numerical}
{Consider the following more general QP problem with both linear equality and inequality constraints:
\begin{equation}
	\label{eq-qp}
	\min_{x'\in \Re^{N}} \Big\{\frac{1}{2} \inprod{x'}{ Q' x'} + \inprod{c'}{x'}  \,\mid\,
		A_E x'   =  b_E, \; A_I x'   \le b_I, \;
		x' \in \cC' \Big\},
\end{equation}
where $Q'\in \cS^N_+$, $c'\in \Re^N$, $ A_E:\Re^{N}\to\Re^{m_E}$,  $ A_I:\Re^N\to\Re^{m_I}$ are two linear maps, $b_E\in \Re^{m_E}$, $b_I\in \Re^{m_I}$ and $\cC':=\{x'\in \Re^N\,\mid\, -\infty\leq\ell'\leq x'\leq u'\leq \infty\}$. By adding a slack variable $s\in \Re^{m_I}$, we can rewrite \eqref{eq-qp} into the following form:
\begin{equation}
\label{eq-qp-standard}
    \min_{x'\in\Re^n, s\in \Re^{m_I}}\left\{ \frac{1}{2} \inprod{x'}{ Q'x'} + \inprod{c'}{x'} \,\mid\, A_E x'  =  b_E, \,A_I x' + s = b_I,\, x' \in \cC', \, s \ge 0 \right\}.
\end{equation}
By taking $n = N + m_I$, $m = m_E + m_I$, $x = (x';s)\in \Re^n$, $\cC:=\{x = (x'; s)\in \Re^n\,\mid\, x'\in \cC,\, s\geq 0\}$, and
\[
    Q := \begin{pmatrix} Q' & 0 \\ 0 & 0\end{pmatrix}\in \cS^{n}_+,\quad c:=\begin{pmatrix} c' \\ 0 \end{pmatrix}\in \Re^{n},\quad  A := \begin{pmatrix} A_E & 0 \\ A_I & I \end{pmatrix}\in \Re^{m\times n}, \quad b := \begin{pmatrix} b_E \\ b_I \end{pmatrix} \in \Re^m,
\]
it is readily seen that problem \eqref{eq-qp-standard} is of the form ({\bf P}) and hence QPPAL can be applied to it directly.

In our numerical experiments, we measure the accuracy of an approximate optimal solution $(x^k,z^k,w^k,y^k) \in \Re^n\times \Re^n \times \cW \times \Re^m$ for QP \eqref{eq-qp-standard} and its dual problem by using the same relative KKT residuals, i.e., $\eta_{p}^k, \eta_{d}^k, \eta_{Q}^k$ and $\eta_{\cC}^k$, as defined in Section \ref{subsec-tau-sigma}. Additionally, we compute the relative gap by
\[
\eta_{g}^k = \frac{\textup{obj}_P-\textup{obj}_D}{1+|\textup{obj}_P|+|\textup{obj}_D|},
\]
where $\textrm{obj}_P := \frac{1}{2}\inprod{x^k}{ Qx^k} + \inprod{c}{x^k}$ and
$\textrm{obj}_D := -\delta_{\cC}^*(-z^k) - \frac{1}{2}\inprod{w^k}{Q w^k} + \inprod{b}{y^k}$.  For any given tolerance ${\tt tol} \geq 0$, we terminate the algorithm if the following condition holds:
\[
	\eta^k:=\max\{ \eta_p^k, \eta_d^k, \eta_Q^k, \eta_C^k, |\eta_g^k| \} \leq {\tt tol}.	
\]
}

{
\subsection{Benchmark} \label{subsection-benchmark}
We compare our Algorithm QPPAL ({implemented in MATLAB, version R2021a}) with the state-of-the-art solvers\footnote{{These solvers are implemented in C/C++, and we call their MATLAB interfaces in our numerical tests.}}, Gurobi, OSQP \cite{OSQP} and QPALM \cite{hermans2019qpalm} for solving various classes of QP problems \eqref{eq-qp} (see Section~\ref{subsec-QP-problems}) whose matrix representations for $ Q $ are available since the latter three solvers all {explicitly require $Q$ to be a matrix}. Moreover, since we use sGS-isPALM as our Phase I algorithm, we also present the numerical results obtained by running sGS-isPALM alone for the purpose of demonstrating the power and importance of our two-phase framework for solving difficult QP problems.
	
All our numerical experiments are conducted on a Linux machine (note that QPALM is only available for Unix machines) with Intel Xeon E5-2650 processors. The following table describes the settings for the termination conditions of each solver in our benchmark. For detailed {descriptions on the meaning of the} settings in Table {tab-settings} for Gurobi, OSQP and QPALM, we refer readers to their documentations.
	
\begin{footnotesize}
\begin{table}[h!]
\centering
\caption{Settings for each solver.}
\label{tab-settings}
\resizebox{0.9\textwidth}{!}{%
\begin{tabular}{|c|c|}
\hline
Solvers    & Settings                                                                                                                                             \\ \hline
GUROBI     & {\tt FeasibilityTol} = {\tt OptimalityTol} = {\tt 1e-6}, \; {\tt TimeLimit} = 10800                                                                  \\ \hline
OSQP       & {\tt eps\_abs} = {\tt eps\_rel} = {\tt eps\_prim\_inf} = {\tt eps\_dual\_inf} = {\tt 1e-6}, \; {\tt max\_iter} = 10000                \\ \hline
QPALM      & {\tt delta} = 100, \; {\tt proximal} = true, \; {\tt scaling} = 2, \; {\tt eps\_abs} =  {\tt eps\_rel} = {\tt 1e-6},\; {\tt max\_iter} = 1000 \\ \hline
sGS-isPALM & {\tt tol} = {\tt 1e-6},\; {\tt MaxIter} = 10000                                                                                                      \\ \hline
QPPAL      & ${\tt tol}_1 = {\tt 1e-4},\; {\tt MaxIter}_1 = 10 \textrm{ or } 1000,\; {\tt tol}_2 = {\tt 1e-6},\; {\tt MaxIter}_2 = 1000$                          \\ \hline
\end{tabular}%
}
\end{table}
\end{footnotesize}
	
Notice that each solver uses its own stopping criteria, so we decide to evaluate the accuracy of the computed solutions using the relative KKT residuals presented previously. We next explain how we conduct our comparisons. Note that all the solvers
{generate primal and dual solutions} of the QP problems. \blue{Thus we} can readily extract the solutions $(x^k, z^k, w^k, y^k)\in \Re^n\times \Re^n \times \Re^n \times \Re^m$ from their outputs. Specifically, for Gurobi, OSQP and QPALM, we always take $w^k = x^k$. Then, we can compute their corresponding relative KKT residuals and objective values, i.e. $\eta_p^k$, $\eta_d^k$, $\eta_Q^k$, $\eta_C^k$, $\eta_g^k$, ${\tt obj}_P$ and ${\tt obj}_D$, respectively.
It is clear that for Gurobi, OSQP and QPALM, it always holds that $\eta_Q^k = 0$.

To compare the robustness of each solver, we
{give our definition of ``failure''} in terms of the KKT residuals, primal and dual objective values. In particular, we say that a solver fails in terms of primal feasibility if it returns a solution with $\eta_p^k > 5\times 10^{-6}$. Similar definition applies to $\eta_d^k$, $\eta_Q^k$, and $\eta_C^k$. For the comparison of objective function values, a solver with ${\tt obj}_P$ and ${\tt obj}_D$ is considered as a failure in terms of primal and dual objective values if
\[
	{\tt err}_p:=\frac{|{\tt obj}_{\rm P} - {\tt obj}_{\rm P, best}|}{1 + | {\tt obj}_{\rm P, best}|} > 5\times 10^{-5},\quad \textrm{and} \quad {\tt err}_d:=\frac{|{\tt obj}_{\rm D} - {\tt obj}_{\rm D, best}|}{1 + | {\tt obj}_{\rm D, best}|} > 5\times 10^{-5},
\]
respectively, where ${\tt obj}_{\rm P, best}$, ${\tt obj}_{\rm D, best}$ being the primal and dual objective values of the solver having the smallest KKT residual $\eta^k$. We then compute the total number of problems that are marked as {failures based on the above definition} for four problem classes in the next subsection.
{
Obviously a solver with a smaller number of failures is considered to be more robust.}

{For comparing the efficiency in terms of computation times}, we adopt the Dolan-Mor\'{e} performance profile \cite{dolan2002benchmarking}~\footnote{\blue{We use the script provided by \url{https://www.mcs.anl.gov/~more/cops/} to create performance profiles.}}. In particular, \blue{suppose that} we are benchmarking $S$ solvers on $P$ problems.  \blue{Let  $t_{i,j}$ denote} the time required by solver $i$ to solve problem $j$. \blue{Then the performance ratio} of solver $i$ for problem $j$ is defined as
	\[
		r_{i,j} = t_{i,j} / \min_{i}\;\{ t_{i,j}:1\leq i\leq S\}, \quad 1\leq i\leq S, \; 1\leq j\leq P.
	\]
	In the performance profile, one plots the following functions for each solver $i$:
	\[
		f_s(\tau):={\frac{1}{P} } \sum_{1\leq j \leq P} f_{\tau}(r_{i,j}),\quad 1\leq i\leq S,\quad \forall \tau\in \Re_{++},
	\]
	where the function $f_{\tau}(\cdot) : \Re \rightarrow [0, 1]$ is defined as
	\[
		f_{\tau}(r):= \left\{
			\begin{array}{ll}
				1 & \textrm{if } r \leq \tau, \\
				0 & \textrm{otherwise,}
			\end{array}
		\right.\quad \forall r\in \Re.
	\]
	It is not difficult to observe that the value of $f_s(\tau)$ represents the fraction of problems which are solved within $\tau$ times of the best solver. Hence, the higher the plotted curve is, the better the performance of the corresponding solver. \blue{Finally, note that the performance profile is meaningful only if all the solvers are able to solve the problems successfully. Hence, failures on a given problem (in the sense that $\eta^k_{kkt}:= \max\left\{\eta_p^k, \eta_d^k, \eta_Q^k, \eta_C^k\right\} > 5\times 10^{-6}$) are represented by ``{\tt NaN}'' (i.e., ``Not a Number'') when
	{we prepare the} performance profiles.}
}

\blue{As argued in \cite{gould2016note}, the performance profiles plots may cause issues when more than two algorithms are compared at the same time. For this reason and for compactness, we also provide the shifted geometric mean (SGM) of the computational time. In particular, for each problem set {consisting of $P$} testing instances, we compute
\[
	\bar t_i = \left(\prod_{1\leq j\leq P}(t_{i,j} + \zeta)\right)^{\frac{1}{P}} - \zeta = \exp\left(\frac{1}{P}\sum_{1\leq j\leq P} \ln(t_{i,j} + \zeta)\right) - \zeta,\quad 1\leq i \leq S.
\]
In this paper, a shift of $\zeta = 1$ is applied. Note that a smaller value of the SGM indicates a better performance.

}

\subsection{Some classes of QP problems}\label{subsec-QP-problems}
We next list {four} classes of QP problems arising from different scenarios with some brief descriptions.

\begin{example}[{\bf QPs arising from relaxations of QAP problems}]
	\label{example-1}
	Given matrices $A,B\in \cS^d$, the quadratic assignment problem (QAP) is given by
	$$
	\min_{{X\in \Re^{d\times d}}}\{ \inprod{\vec(X)}{(B\otimes A)\vec(X)} \mid Xe = e = X^Te, X \geq 0, X \in\{0,1\}^{d\times d}\},
	$$
	where $\otimes$ denotes the Kronecker product, $\vec(X)$ is the vectorization of the matrix $X$, i.e.,
	\[
	    \vec(X) = [x_{1,1},\ldots,x_{d,1},x_{1,2},\ldots,x_{d,2},\ldots,x_{1,d},\ldots,x_{d,d}]^{T}{\in \Re^{d^2}}.
	\]
	It has been shown in \cite{QPbound01} that a reasonably good lower bound for the above QAP can often be obtained by solving the following convex QP relaxation:
	\begin{eqnarray}
		\min_{{\vec(X)\in \Re^{d^2}}} \{ \inprod{\vec(X)}{ Q\vec(X)} \mid (e^T \otimes I )\vec(X) = e = (I\otimes e^T)\vec(X), \vec(X) \geq 0 \},
		\label{eq-QAP-QP}
	\end{eqnarray}
	where $ Q = B\otimes A - I\otimes S - T\otimes I$, and $S,T\in\cS^d$ are given as follows. Consider the eigenvalue decompositions, $A = V_A D_A V_A^T$, $B=V_B D_B V_B^T$, where $V_A$ and $D_A=\diag{(\alpha_1,\dots,\alpha_d)}$ correspond to the eigenvectors and eigenvalues of $A$, and $V_B$ and $D_B = \diag{(\beta_1,\dots,\beta_d)}$ correspond to the eigenvectors and eigenvalues of $B$, respectively. We assume that $\alpha_1 \geq \ldots \geq \alpha_d$ and $\beta_1 \leq \ldots \leq \beta_d$. Let $(\bar{s},\bar{t})$ be an optimal solution to the LP: $\max \{ e^T s + e^T t\mid s_i + t_j \leq \alpha_i\beta_j,\; i,j=1,\dots,d\}$, whose solution can be computed analytically as shown in \cite{QPbound01}. Then $S = V_A\diag{(\bar{s})}V_A^T$ and $T=V_B \diag{(\bar{t})} V_B^T$. The data for the QAPs are obtained from {\tt QAPLIB} \cite{QAPLIB}.
\end{example}

{
\begin{example}[{\bf QP relaxations of quartic optimization with binary constraints}]
	\label{example-2}
	Consider the following binary constrained quartic optimization problem
	\begin{equation*}
		\min_{x\in \mathbb{R}^d}\left\{ \left(\left\langle x, Qx \right \rangle\right)^2\,\mid \, x\in \{0,1\}^{d},\; \norm{x}_0 \geq \beta\right\},
	\end{equation*}
	where $Q\in \cS^{d}$, and $\beta >0$ is a given constant. The presence of the cardinality constraint $\norm{x}_0 \geq \beta$ is to avoid the trivial case when $x =0$ is an optimal solution. Let $\lambda_{max}(Q)$ and $\lambda_{min}(Q)$ denote the maximal and minimal eigenvalue of $Q$, respectively. It holds that the minimal eigenvalue of $Q\otimes Q$ is given as follows:
	\[
		\lambda_{min}(Q\otimes Q) = \left\{
				\begin{array}{ll}
					\lambda_{min}(Q)^2 & \textrm{ if } \lambda_{min}(Q) \geq 0, \\
					\lambda_{max}(Q)^2 & \textrm{ if } \lambda_{max}(Q) \leq 0, \\
					\lambda_{min}(Q)\cdot \lambda_{max}(Q) & \textrm{ otherwise}. \\
				\end{array}
		\right.
	\]
	Notice that the $ \left(\left\langle x, Qx \right \rangle\right)^2 =  \left\langle \vec(xx^T), (Q\otimes Q) \vec(xx^T) \right \rangle $. Then, a direct QP relaxation for the above optimization problem is given as follows:
	\begin{equation*}
		\min_{X\in \cS^{d}, x\in \Re^{d}}\left\{ \left\langle \svec(X), \tilde{Q} \svec(X) \right \rangle + \lambda_0\left\langle e_d, x\right\rangle \,\mid\, (X, x)\in \mathcal{F}_{BIQ} \right\},
	\end{equation*}
	where $\lambda_0 := \min\{0, \lambda_{min}(Q\otimes Q)\}$,
	{$e_d\in\Re^d$ is the vector of all ones},
	$\tilde{Q} \in \cS_+^{d(d+1)/2}$ satisfies
	\[
		\left\langle \vec(X), \left[(Q\otimes Q) - \lambda_0 I_{d^2} \right] \vec(X)\right\rangle = \left\langle \svec(X), \tilde{Q} \svec(X)\right\rangle,
		\]
	for any $X\in \cS^d$, and
	\begin{equation*}
		\mathcal{F}_{BIQ}:= \left\{ (X, x)\in \cS^{d}\times \Re^d\left\lvert \;
		\begin{aligned}
			&\mathrm{diag}(X) = x,\; \left\langle e_d, x\right\rangle \geq \beta,\;-X_{ij} + x_j\geq 0,\; -X_{ij} + x_i\geq 0,\\
			&X_{ij} - x_i - x_j\geq -1, \;1\leq i<  j\leq d,\; X\geq 0,\; x\geq 0
		\end{aligned}\right.
			\right\}.
	\end{equation*}
	{Note that in the above, $\svec(X)$ denotes the vector obtained by stacking
	the upper triangular part of $X$, column by column sequentially. Moreover, the off-diagonal
	elements of $X$ are scaled by $\sqrt{2}$.}
	
	In our numerical experiments, the tested matrices $Q$ are selected from the BIQMAC library \cite{BIQMAC} and we take $\beta = d/5$.
\end{example}
}

\begin{example}[{\bf QPs selected {from} Maros-M\'{e}sz\'{a}ros collection}]
	\label{example-3}
	In this example, we compare the performance of all the solvers on the QP instances that are selected from the Maros-M\'{e}sz\'{a}ros collections \cite{Maros1999}. The QP problems from this collection are often used to benchmark QP solvers since this collection contains many large-scale and very difficult (ill-conditioned) QP problems. Thus, they are quite challenging to solve. {Since we focus on solving large-scale QP problems in this paper, we only select those problems such that the sum of linear equality and inequality constraints and the number of decision variables is at least 1000.}
\end{example}

\begin{example}[{\bf QPs arsing from portfolio optimization}]
	\label{example-4}
	Portfolio optimization \cite{Perold1984,Takehara1993} employed by the investment community seeks to allocate asserts in a way that optimizes the risk adjusted return. In this example, we consider a simplified version of portfolio optimization which is in fact a convex QP given as follows:
	\begin{equation*}
		\min_{{x\in \Re^n}} \; \left\{\gamma \left\langle x, \Sigma x\right \rangle -
		\left \langle \mu, x\right\rangle \mid \left \langle e_n, x\right\rangle = 1, \; x\geq 0\right\},
	\end{equation*}
	where $x\in \Re^n$ is the decision variable, and the data matrix $\Sigma \in \cS^n$ is symmetric positive semidefinite, $\mu \in \Re^n$, $\gamma >0$ and $e_n \in \Re^n$ is the vector of all ones. We generate our test data randomly via the following {\sc Matlab} script as follows:
	\begin{verbatim}
	    n = 1000*k; m = 10*k;
	    F = sprandn(n, m, 0.1); D = sparse(diag(sqrt(m)*rand(n,1)));
	    Sigma = cov(F') + D;
	    mu = randn(n,1); gamma = 1.0;
	\end{verbatim}
\end{example}

\subsection{Numerical results} \label{subsec-numercial-results}
{
\blue{We are now ready to present our numerical results by showing the performance profiles, the failure rates and the shifted geometric mean (SGM) of the five solvers on the four previous examples. For the rest of this section, the solvers ``a'', ``b'', ``c'', ``d'' and ``e'' represent GUROBI, OSQP, QPALM, sGS-isPALM and QPPAL, respectively.} The detailed computational results are {available} in the supplementary material.

Fig. \ref{fig:comparison-QAP} and Fig. \ref{fig:comparison-BIQ} present the computational results for Example \ref{example-1} and Example \ref{example-2}, respectively. We observe from the performance profiles that our proposed algorithms are indeed much more efficient than the existing state-of-the-art solvers. In particular, QPPAL is shown to be substantially more efficient than other solvers. While OSQP performs better than Gurobi on Example \ref{example-2}, the former is slightly less efficient than the latter on Example \ref{example-1}. For both examples, QPALM turns out to be less efficient than the {other} solvers. This observation supports our claim that for solving large-scale dense QP problems, {applying the proximal ALM method to the {(restricted-Wolfe)} dual problem {\bf (D)} is more effective.} For the comparison on the failure numbers of each solver, we see that QPPAL and sGS-isPALM {are able to solve all the instances successfully
for both examples, whereas OSQP and QPALM can solve all the instances in Example \ref{example-2} but fail on quite a few instances} in Example \ref{example-1}. Gurobi is also quite robust as it can solve all the instances in both examples except
{for three instances with} bad $\eta_C$ in Example \ref{example-1}.
Overall, we see that QPPAL and sGS-isPALM are quite efficient and robust for solving large-scale and dense QP problems.
}

\begin{figure*}[ht]
	\centering
	\begin{minipage}{0.48\linewidth}
		\includegraphics[width=\linewidth]{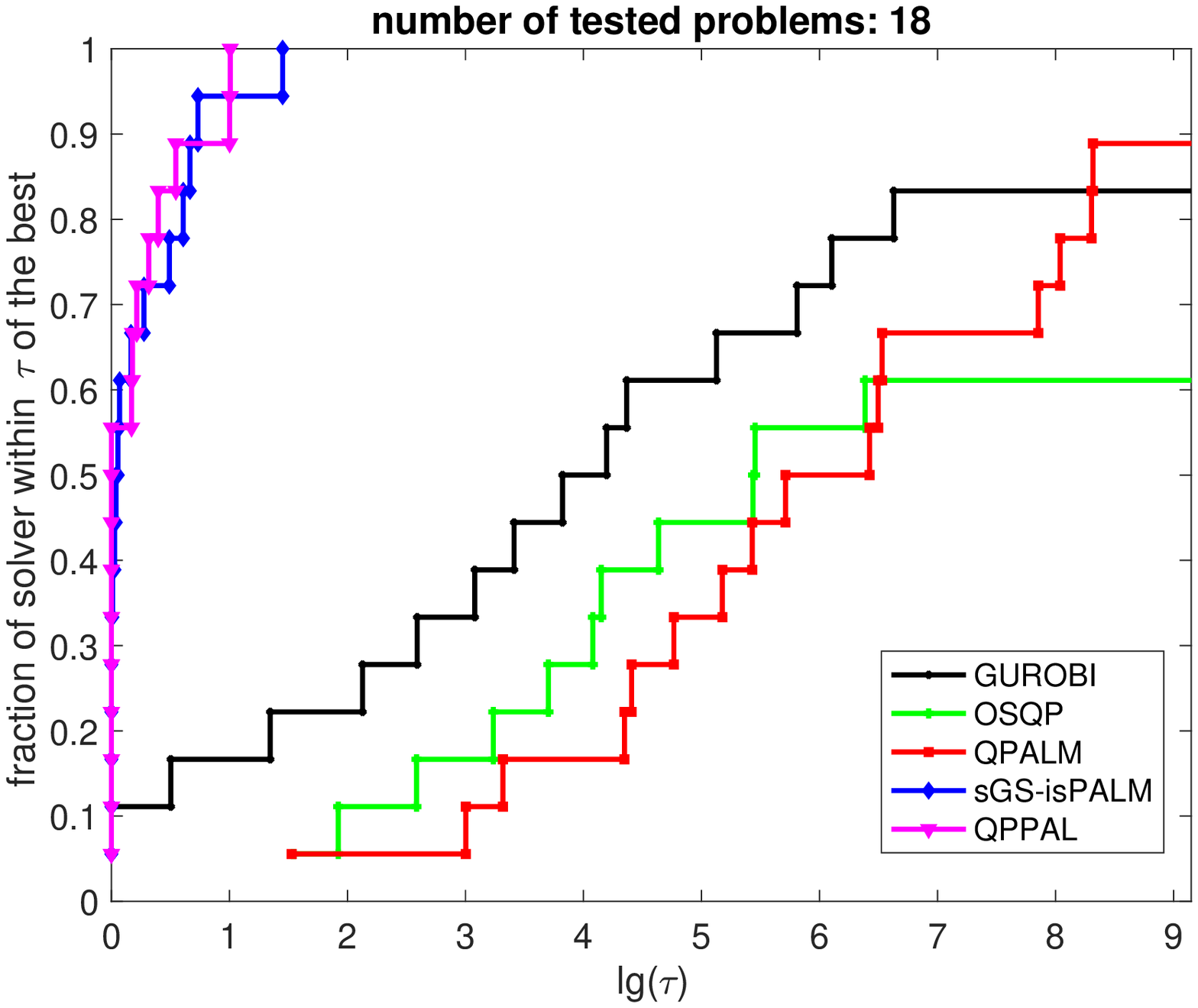}
	\end{minipage}
	\hfill
	\begin{minipage}{0.48\linewidth}
	\begin{tabular}{|l|l|l|l|l|l|}
		\hline
		Solver & a & b & c & d & e \\ \hline
		$\eta_p$ & 0 & 0 & 0 & 0 & 0 \\ \hline
		$\eta_d$ & 0 & 5 & 2 & 0 & 0 \\ \hline
		$\eta_Q$ & 0 & 0 & 0 & 0 & 0 \\ \hline
		$\eta_C$ & 3 & 2 & 1 & 0 & 0 \\ \hline
		${\tt obj}_p$ & 0 & 0 & 0 & 0 & 0 \\ \hline
		${\tt obj}_d$ & 0 & 0 & 0 & 0 & 0 \\ \hline
		$\eta_{kkt}$ & 3 & 7 & 2 & 0 & 0 \\ \hline
		{\tt Failure rate [\%]} & 16.67 & 38.89 & 11.11 & 0.00 & 0.00 \\ \hline
		{\tt SGM [sec]} & 38.09 & 59.75 & 203.20 & 4.45 &  4.32 \\ \hline
	\end{tabular}
%
	\end{minipage}
	\caption{Performance comparison for Example \ref{example-1}. {\bf Left}: Performance profile. {\bf Right}: Statistics for each solver.}
	\label{fig:comparison-QAP}
\end{figure*}

\begin{figure*}[ht]
	\centering
	\begin{minipage}{0.48\linewidth}
		\includegraphics[width=\linewidth]{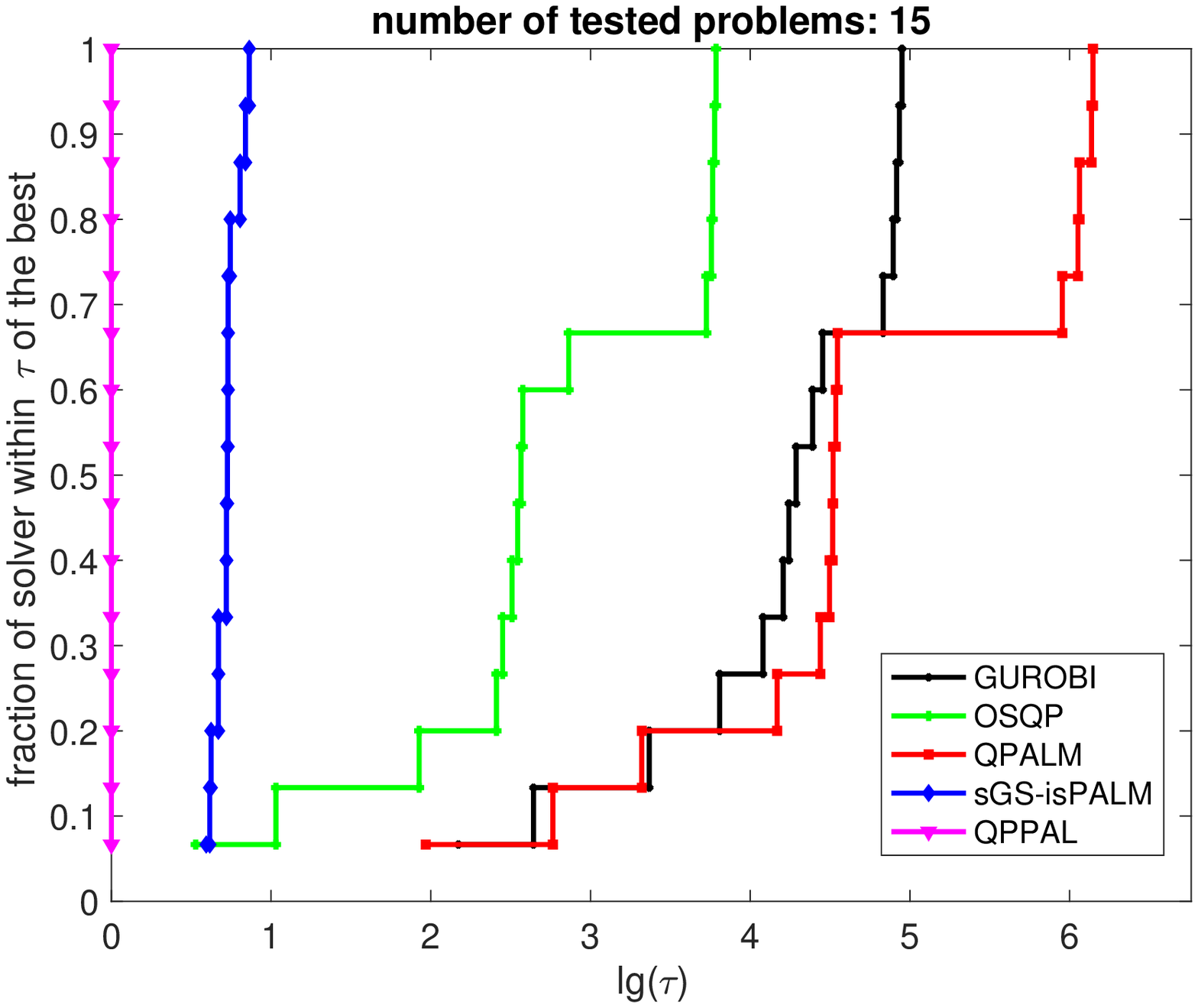}
	\end{minipage}
	\hfill
	\begin{minipage}{0.48\linewidth}
	\begin{tabular}{|l|l|l|l|l|l|}
		\hline
		Solver & a & b & c & d & e \\ \hline
		$\eta_p$ & 0 & 0 & 0 & 0 & 0 \\ \hline
		$\eta_d$ & 0 & 0 & 0 & 0 & 0 \\ \hline
		$\eta_Q$ & 0 & 0 & 0 & 0 & 0 \\ \hline
		$\eta_C$ & 0 & 0 & 0 & 0 & 0 \\ \hline
		${\tt obj}_p$ & 0 & 0 & 0 & 0 & 0 \\ \hline
		${\tt obj}_d$ & 0 & 0 & 0 & 0 & 0 \\ \hline
		$\eta_{kkt}$ & 0 & 0 & 0 & 0 & 0 \\ \hline
		{\tt Failure rate [\%]} & 0.00 & 0.00 & 0.00 & 0.00 & 0.00 \\ \hline
		{\tt SGM [sec]} & 127.30 & 49.70 & 179.87 & 11.91 &  7.27 \\ \hline
	\end{tabular}
	\end{minipage}
	\caption{Performance comparison for Example \ref{example-2}. {\bf Left}: Performance profile. {\bf Right}: Statistics for each solver.}
	\label{fig:comparison-BIQ}
\end{figure*}

{
Fig. \ref{fig:comparison-Maros} shows the comparison results for solving QPs in Example \ref{example-3}. Recall that these QP problems are highly sparse and are quite challenging to solve {because of degeneracy such as the nonexistence of a strictly feasible point
for either the primal or dual problem}. It is commonly accepted that interior point methods would be both efficient and robust {for solving degenerate problems}. The numerical results for Gurobi that we have presented support this argument.
}

\begin{figure*}[ht]
	\centering
	\begin{minipage}{0.48\linewidth}
		\includegraphics[width=\linewidth]{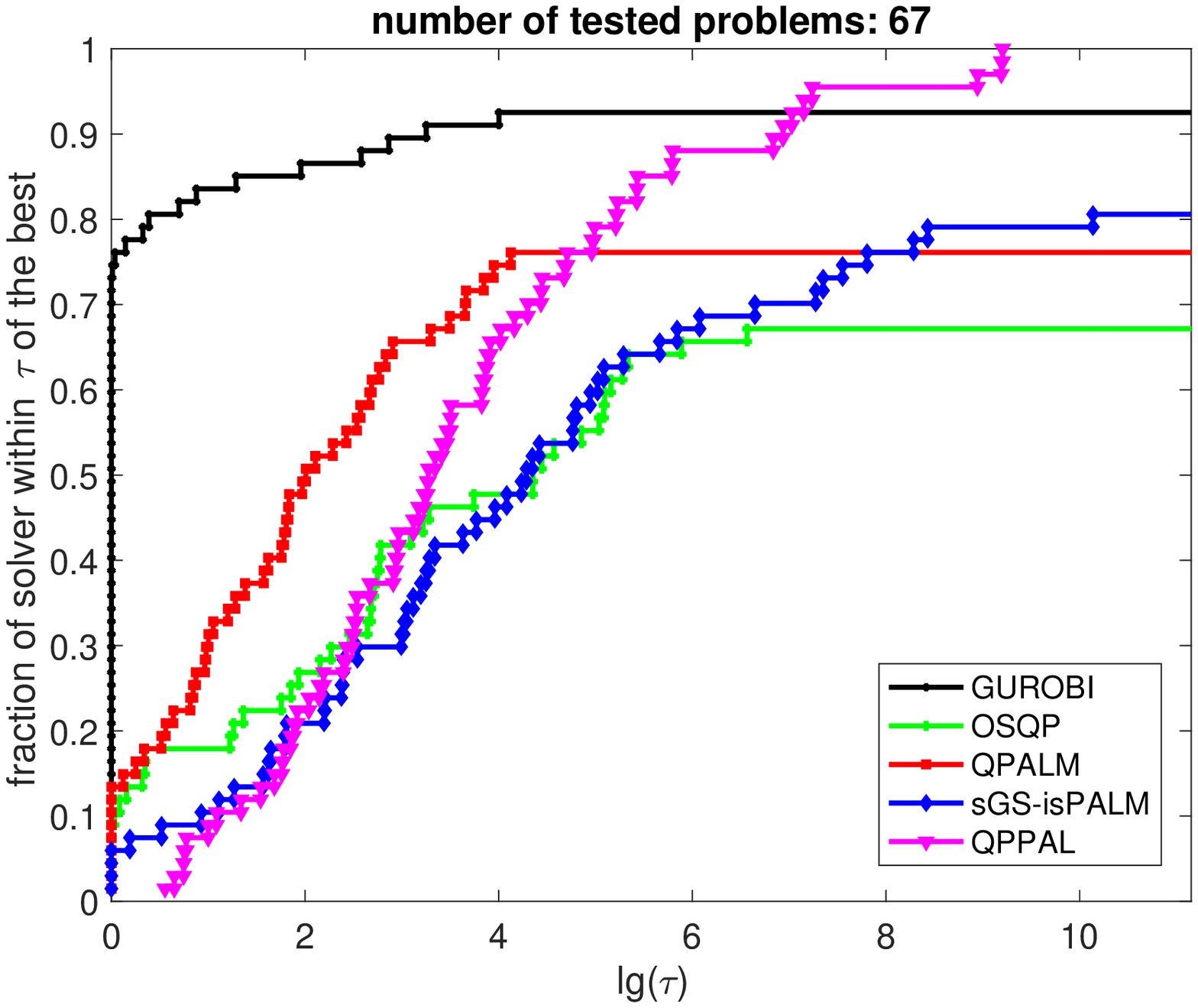}
	\end{minipage}
	\hfill
	\begin{minipage}{0.48\linewidth}
	\begin{tabular}{|l|l|l|l|l|l|}
		\hline
		Solver & a & b & c & d & e \\ \hline
		$\eta_p$ & 0 & 7 & 10 & 2 & 0 \\ \hline
		$\eta_d$ & 0 & 20 & 7 & 13 & 0 \\ \hline
		$\eta_Q$ & 0 & 0 & 1 & 10 & 0 \\ \hline
		$\eta_C$ & 5 & 12 & 4 & 9 & 0 \\ \hline
		${\tt obj}_p$ & 0 & 24 & 16 & 24 & 14 \\ \hline
		${\tt obj}_d$ & 0 & 24 & 16 & 25 & 14 \\ \hline
		$\eta_{kkt}$ & 5 & 22 & 16 & 13 & 0 \\ \hline
		{\tt Failure rate [\%]} & 7.46 & 32.84 & 23.88 & 19.40 & 0.00 \\ \hline
		{\tt SGM [sec]} & 0.38 & 2.38 & 1.37 & 4.40 &  2.97 \\ \hline
	\end{tabular}
	\end{minipage}
	\caption{Performance comparison for Example \ref{example-3}. {\bf Left}: Performance profile. {\bf Right}: Statistics for each solver.}
	\label{fig:comparison-Maros}
\end{figure*}

{
Note that though the focus of this paper is to design scalable and robust algorithms for solving large-scale QP problems having some or all the three characteristics mentioned in Section \ref{intro}, our purpose of presenting the results for those QP problems in the Maros-M\'{e}sz\'{a}ros collection is to evaluate the robustness of our proposed algorithm. Indeed the presented results show that our proposed QPPAL is robust even for difficult sparse QPs. More specifically, the results show the need of our Phase II algorithm since one can easily observe that there are numerous problems in the collection for which the first-order algorithms OSQP and sGS-isPALM cannot deliver accurate approximate solutions. However, our two phase algorithm QPPAL is able to obtain solutions with the desired accuracy. As a second-order algorithm,
{QPALM is also demonstrated to be more robust than OSQP and sGS-isPALM.
However, it is} less robust than Gurobi and QPPAL. For the comparison between OSQP and sGS-isPALM, we observe that both algorithms have similar performance in terms of robustness.
}

{
For the comparison on efficiency, it is clear that Gurobi is the most powerful solver that outperforms all the other solvers in terms of the computational time on the successful instances. Moreover, QPALM is generally more efficient than OSQP and our solvers in this collection. We also observe that the numerical results of Gurobi, OSQP and QPPAL in this paper coincides the numerical results presented in \cite[Section 7.3]{hermans2020qpalm}. In addition, given the fact that QPALM and OSQP are less robust than Gurobi, it may take more computational effort for QPALM and OSQP to compute more accurate solutions. Therefore, we would recommend Gurobi for solving sparse and difficult QP problems if one wants to obtain highly accurate solutions efficiently. On the other hand, QPPAL is less efficient than QPALM but it has comparable performance as OSQP and it is more efficient than sGS-isPALM. The latter observation together with the comparison of robustness between QPPAL and sGS-isPALM again shows that the second phase of QPPAL is indeed necessary and further supports our motivation to design a two-phase algorithm. For the comparison between OSQP and sGS-isPALM, one sees that OSQP usually requires less computational time, especially for the small-scale problems.
}

{
Fig. \ref{fig:comparison-Portfolio} presents the computational results for Example \ref{example-4}. From the presented results, we can see that QPPAL outperforms the other solvers in terms of the computational time. In fact, QPPAL is at least ten times faster than the other solvers when the problem size is large, say when the dimension of the decision variable is larger than 2000.
One can also observe that both OSQP and QPALM require much longer computational times and the objective values returned by OSQP and QPALM are obviously worse than Gurobi, sGS-isPALM and QPPAL. One possible reason is that the matrix $Q$ in these problems are fully dense and applying a direct solver for solving {linear systems} involving $Q$ may not be a wise choice. However, OSQP and QPALM currently only support direct solvers.
}

\begin{figure*}[ht]
	\centering
	\begin{minipage}{0.48\linewidth}
		\includegraphics[width=\linewidth]{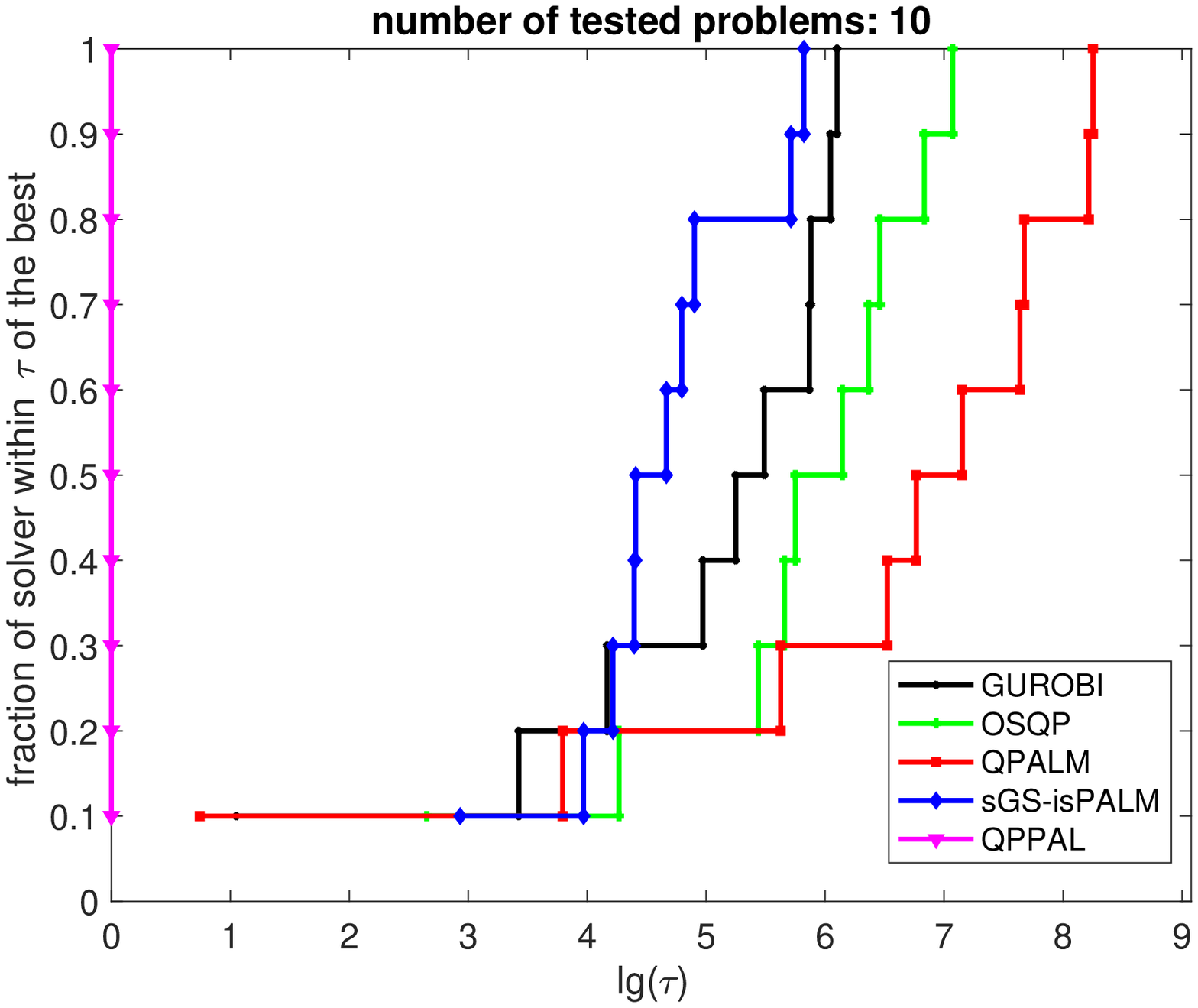}
	\end{minipage}
	\hfill
	\begin{minipage}{0.48\linewidth}
	\begin{tabular}{|l|l|l|l|l|l|}
		\hline
		Solver & a & b & c & d & e \\ \hline
		$\eta_p$ & 0 & 0 & 0 & 0 & 0 \\ \hline
		$\eta_d$ & 0 & 0 & 0 & 0 & 0 \\ \hline
		$\eta_Q$ & 0 & 0 & 0 & 0 & 0 \\ \hline
		$\eta_C$ & 0 & 0 & 0 & 0 & 0 \\ \hline
		${\tt obj}_p$ & 0 & 10 & 10 & 0 & 0 \\ \hline
		${\tt obj}_d$ & 0 & 10 & 10 & 0 & 0 \\ \hline
		$\eta_{kkt}$ & 0 & 0 & 0 & 0 & 0 \\ \hline
		{\tt Failure rate [\%]} & 0.00 & 0.00 & 0.00 & 0.00 & 0.00 \\ \hline
		{\tt SGM [sec]} & 80.15 & 132.96 & 216.48 & 61.93 & 3.12 \\ \hline
	\end{tabular}
	\end{minipage}
	\caption{Performance comparison for Example \ref{example-4}. {\bf Left}: Performance profile. {\bf Right}: Statistics for each solver.}
	\label{fig:comparison-Portfolio}
\end{figure*}

{
We finish this subsection with some final comments on the computational results. Obviously, Gurobi and QPALM have the best performance for convex QP problems whose matrices {$Q$ and $A$} are highly sparse and with special structures. However,
QPALM is demonstrated to be less robust than Gurobi. On the contrary, for dense and large-scale problems, Gurobi and QPALM may no longer be a good option since they require too much computational effort. As ADMM-type algorithms, both OSQP and sGS-isPALM are highly efficient for well-conditioned problems such as convex QP relaxations of QAP and
{quartic polynomial minimization} problems. However, our numerical results show that sGS-isPALM is more likely to outperform OSQP in terms of efficiency and accuracy for large-scale QP problems. On the other hand, for small-scale QP problems, sGS-isPALM is demonstrated to have comparable performance as OSQP in terms of robustness. Finally, the computational results for those difficult problems that sGS-isPALM cannot solve efficiently indicate that the design of our two-phase algorithm {QPPAL} is indeed useful and important. Overall, the above numerical results demonstrate that the proposed QPPAL is highly {efficient} and robust for solving large-scale and dense convex QP problems.
}

\subsection{Computational results on matrices \texorpdfstring{$Q$}{Q} without matrix representations}

In this subsection, we consider QP problems arising in Section~\ref{subsec-QP-problems} for which the matrix representations for $Q$ may not be available. The test problems are selected from QP relaxations for QAP problems with $ d \geq 100 $ and BIQ problems with $ n\geq 150 $. For these problems, the matrices $ Q $ are usually fully dense. Moreover, even if the matrix representations for $ Q $ are available, storing them would require a large amount of memory. Thus, we can only use iterative solvers (such as PSQMR) to solve the underlying linear systems, and Gurobi, QPALM and OSQP are not able to cope with these large scale QPs since they currently only support direct solvers for solving linear systems.

{
The computational results are presented in Table \ref{tab-noQ}. In the table, the three integers below the name of the problem represent the sizes of the underlying QP problem. For instance, $(200, 0, 10000)$ under the first column indicates that the QP problem has $200$ linear equality constraints, no linear inequality constraint and 10,000 decision variables. The columns {\tt Iter} and {\tt Time} represent the iteration number and computational time, respectively. In particular, $5(280)$ means that QPPAL takes 5 phase II iterations with 280 phase I iterations for warm-starting. The meanings of the remaining columns in the table are clear from the context {and we} omit the description for simplicity.
}

From the {Table \ref{tab-noQ}}, we observe that for very large-scale and dense QP problems, our QPPAL is still able to solve them efficiently and robustly.
{In particular, for the last instance in the table, QPPAL is able to solve
the QP with more than 374 thousand constraints and 125 thousand variables in under 5 minutes.}
This shows that our proposed algorithm is indeed scalable, robust and highly efficient for convex QP problems having the three characteristics mentioned in Section \ref{intro}. {Moreover, QPPAL is empirically more efficient than sGS-isPALM}. This again supports our motivation of developing QPPAL in this paper.

\begin{longtable}[c]{|l|l|l|r|l|l|l|l|r|r|r|}
\caption{Numerical results for QP problems in Example \ref{example-1} and \ref{example-2} having no matrix representation of $Q$.}
\label{tab-noQ}\\ \hline
Name & Solver & Iter & Time & $\eta_p$ & $\eta_d$ & $\eta_Q$ & $\eta_C$ & ${\tt obj}_P$ & ${\tt obj}_D$ & $\eta_g$ \\ \hline
\endfirsthead
\multicolumn{11}{c}
{{\bfseries Table \thetable\ continued from previous page}} \\ \hline
Name &  & Iter & Time & $\eta_p$ & $\eta_d$ & $\eta_Q$ & $\eta_C$ & ${\tt obj}_P$ & ${\tt obj}_D$ & $\eta_g$ \\ \hline
\endhead
\hline
\endfoot
tai100a & sGS-isPALM & 567 & 2.577 & 7e-13 & 1e-06 & 3e-07 & 7e-07 & 4.5239e-01 & 4.5239e-01 & -1e-07 \\
(200,0,10000) & QPPAL & 8(278) & 2.122 & 6e-16 & 7e-07 & 3e-08 & 6e-08 & 4.5239e-01 & 4.5239e-01 & -1e-07 \\ \hline
tai100b & sGS-isPALM & 1870 & 16.619 & 6e-13 & 1e-06 & 2e-07 & 9e-07 & 2.7470e-01 & 2.7470e-01 & -3e-07 \\
(200,0,10000) & QPPAL & 11(237) & 5.727 & 6e-16 & 8e-07 & 4e-08 & 8e-08 & 2.7470e-01 & 2.7470e-01 & -3e-07 \\ \hline
tai256c & sGS-isPALM & 280 & 4.096 & 3e-07 & 1e-06 & 8e-08 & 2e-08 & 8.4942e-02 & 8.4942e-02 & 6e-08 \\
(512,0,65536) & QPPAL & 8(185) & 5.139 & 6e-16 & 3e-07 & 2e-09 & 7e-12 & 8.4942e-02 & 8.4942e-02 & 4e-10 \\ \hline
tai150b & sGS-isPALM & 2183 & 44.252 & 6e-13 & 9e-07 & 2e-07 & 9e-07 & 2.7753e-01 & 2.7753e-01 & -3e-07 \\
(300,0,22500) & QPPAL & 19(307) & 23.315 & 7e-16 & 8e-07 & 2e-09 & 2e-10 & 2.7753e-01 & 2.7753e-01 & -2e-07 \\ \hline
tho150 & sGS-isPALM & 2316 & 44.249 & 6e-13 & 1e-06 & 2e-07 & 7e-07 & 3.7766e-01 & 3.7766e-01 & -2e-07 \\
(300,0,22500) & QPPAL & 21(320) & 24.713 & 7e-16 & 9e-07 & 3e-09 & 5e-10 & 3.7766e-01 & 3.7766e-01 & -2e-07 \\ \hline
wil100 & sGS-isPALM & 2630 & 19.507 & 8e-13 & 9e-07 & 2e-07 & 1e-06 & 5.2368e-01 & 5.2368e-01 & -2e-07 \\
(200,0,10000) & QPPAL & 14(299) & 8.132 & 9e-16 & 7e-07 & 5e-08 & 5e-08 & 5.2368e-01 & 5.2368e-01 & -2e-07 \\ \hline
esc128 & sGS-isPALM & 244 & 1.398 & 1e-06 & 7e-07 & 5e-07 & 1e-08 & 1.2167e-01 & 1.2167e-01 & 3e-08 \\
(256,0,16384) & QPPAL & 3(154) & 1.033 & 4e-16 & 1e-07 & 3e-07 & 5e-11 & 1.2167e-01 & 1.2167e-01 & 2e-11 \\ \hline
be150.3.1 & sGS-isPALM & 266 & 25.502 & 8e-08 & 3e-07 & 1e-06 & 3e-11 & 2.8738e+06 & 2.8738e+06 & -3e-07 \\
(150,33526,11475) & QPPAL & 11(176) & 16.459 & 8e-07 & 8e-07 & 2e-07 & 6e-16 & 2.8738e+06 & 2.8738e+06 & 1e-07 \\ \hline
be150.8.1 & sGS-isPALM & 266 & 41.347 & 8e-08 & 3e-07 & 1e-06 & 1e-11 & 7.2801e+06 & 7.2801e+06 & -3e-07 \\
(150,33526,11475) & QPPAL & 11(176) & 26.238 & 8e-07 & 7e-07 & 3e-07 & 2e-16 & 7.2801e+06 & 7.2801e+06 & 1e-07 \\ \hline
be200.3.1 & sGS-isPALM & 266 & 47.742 & 6e-08 & 3e-07 & 1e-06 & 2e-11 & 5.1907e+06 & 5.1907e+06 & -7e-08 \\
(200,59701,20300) & QPPAL & 14(176) & 32.482 & 6e-07 & 1e-07 & 2e-07 & 7e-16 & 5.1907e+06 & 5.1907e+06 & 5e-08 \\ \hline
be250.1 & sGS-isPALM & 270 & 52.404 & 5e-08 & 4e-07 & 9e-07 & 5e-11 & 3.1315e+06 & 3.1315e+06 & 1e-07 \\
(250,93376,31625) & QPPAL & 16(177) & 38.030 & 1e-07 & 6e-07 & 2e-07 & 2e-15 & 3.1315e+06 & 3.1315e+06 & -8e-07 \\ \hline
bqp250-1 & sGS-isPALM & 270 & 53.029 & 5e-08 & 4e-07 & 1e-06 & 1e-11 & 1.0644e+07 & 1.0644e+07 & 9e-08 \\
(250,93376,31625) & QPPAL & 16(177) & 38.077 & 9e-08 & 6e-07 & 9e-08 & 7e-16 & 1.0644e+07 & 1.0644e+07 & -7e-07 \\ \hline
bqp500-1 & sGS-isPALM & 440 & 488.023 & 3e-09 & 1e-06 & 1e-06 & 6e-12 & 4.0835e+07 & 4.0835e+07 & -3e-07 \\
(500,374251,125750) & QPPAL & 18(195) & 274.971 & 4e-07 & 3e-08 & 2e-07 & 1e-15 & 4.0835e+07 & 4.0835e+07 & -1e-06 \\ \hline
\end{longtable}

\section{Conclusions}\label{sec-conclusion}
In this paper, we have proposed a two-phase proximal augmented Lagrangian method (QPPAL) for solving convex quadratic programming problems. In the first phase of QPPAL, we applied {a} symmetric Gauss-Seidel based semi-proximal augmented Lagrangian method for the purpose of generating a good starting point. In the second phase of QPPAL, a proximal augmented Lagrangian method of multipliers with elegant convergence properties developed recent by Li et al.~\cite{Li2019LP} was applied. To solve the corresponding inner subproblems efficiently, a semismooth Newton method with a fast local convergence rate was adopted. With well-developed theoretical results, we then conducted extensive numerical experiments to evaluate the performance of the proposed algorithm against the highly powerful commercial solver Gurobi, the operator splitting based solver OSQP, {and the primal based augmented Lagrangian method (QPALM)}. Promising numerical results demonstrated that the proposed QPPAL is highly efficient and robust for solving large-scale and dense problems. {Lastly, for further efficiency, we plan to implement our proposed algorithms by using C/C++ in our future work.}

\section*{Acknowledgments}
{We thank the editor and referees for their valuable suggestions and insightful comments, which have helped to improve the quality of this paper.} \blue{We also thank Miss Can Wu at the Hong Kong Polytechnic University for her comments on the revised version of this paper.}

\appendix

\section{Appendix}
\label{sec-appendix}

\subsection{Proof of Proposition \ref{prop:equi-scb-padmm}} \label{subsec-pf-padmm}
Before proving the theorem, we need the following lemma that provides an estimation of the distance between $v^{k+1}$ and $\bar{v}^{k+1}$ in the Algorithm isPALM.
\begin{lemma}\label{lemma-distest}
	Let $\{(v^k,x^k)\}$ be the sequence generated by the Algorithm {isPALM} and $\{\bar v^k\}$ be defined by \eqref{aug-v}. Then,
	\[\norm{v^{k+1}-\bar v^{k+1}}_{\cN} \le \norm{\cN^{-1/2}d^k}\le \varepsilon_k, \quad \forall k\ge 0.\]
\end{lemma}
\begin{proof}
	From \eqref{aug-v} and \eqref{aug-d},  we have for all $k\ge 0$,
	\begin{align*}
		0\in{}&\partial g(\bar v^{k+1}) + \cG x^k + \sigma\cG(\cG^*\bar v^{k+1} - c) + \cT(\bar v^{k+1} - v^k),\\
		d^k\in{}&\partial g(v^{k+1}) + \cG x^k + \sigma\cG(\cG^* v^{k+1} - c) + \cT(v^{k+1} - v^k).
	\end{align*}
	Then, by \eqref{monosub2}, we know that
	\[\inprod{d^k + (\cT + \sigma\cG\cG^*{)}(v^{k+1} - \bar v^{k+1})}{\bar v^{k+1} - v^{k+1}} \ge \norm{\bar v^{k+1} - v^{k+1}}^2_{\Sigma_g}.\]
	By simple calculations, we can obtain that
	\[\norm{\bar v^{k+1} - v^{k+1}}^2_{\Sigma_g + \cT + \sigma\cG\cG^*}\le \inprod{d^k}{\bar v^{k+1} - v^{k+1}},\] i.e.,
	\[\norm{\bar v^{k+1} - v^{k+1}}^2_{\cN}\le \inprod{\cN^{-1/2}d^k}{\cN^{1/2}(\bar v^{k+1} - v^{k+1})}\le \norm{\cN^{-1/2}d^k}\norm{\bar v^{k+1} - v^{k+1}}_{\cN}.\]
	From here, the required result follows directly.
\end{proof}
\begin{proof}[Proof of Proposition \ref{prop:equi-scb-padmm}]
	The positive definiteness of $\widehat\cE$ and the equivalence follows directly from \cite[Theorem 4.1]{ChenipALM}. By Lemma \ref{lemma-distest}, we know that for $k\ge 0$
	\[
	\norm{\widehat\cE^{-1/2}d^k}\le \norm{\cE_d^{-1/2}}\norm{\hat\delta^k - \delta^k} + \norm{\widehat\cE^{-1/2}}\norm{\delta^k}\le \big((2p-1)\norm{\cE_d^{-1/2}}+p\norm{\widehat\cE^{-1/2}}\big)\epsilon_k,
	\]
	which completes the proof.
\end{proof}

\subsection{Proof of Lemma~\ref{lem-Pk}}\label{subsec-lem-Pk}
\begin{proof}
	By the first optimality conditions for the minimax problem in~\eqref{eq-Pk}, we derive that, for any $ (w,y,x)\in P_k(\bar{w},\bar{y},\bar{x}) $
	\begin{eqnarray*}
		0 & = & Q(w-x)+\frac{\tau_k}{\sigma_k}Q(w - \bar{w}), \\
		0 & = & -b + \blue{Ax}+ \frac{\tau_k}{\sigma_k}(y - \bar{y}), \\
		0 & \in & Qw - A^*y+ c+ \partial \delta_{\cC}(x) + \frac{1}{\sigma_k} (x - \bar{x}).
	\end{eqnarray*}
	By the definition of $ \mathcal{T}_{\tilde{l}} $, the above conditions can be written as
	\begin{equation*}
		(\mathcal{T}_{\tilde{l}} + \frac{1}{\sigma_k}\Lambda_k)(w,y,x) - \frac{1}{\sigma_k}\Lambda_k(\bar{w}, \bar{y},\bar{x}) = 0,\quad \forall\,(w,y,x)\in P_k(\bar{w},\bar{y},\bar{x}).
	\end{equation*}
	This establishes~\eqref{eq-lem-Pk}. The last statement in the lemma follows easily by~\eqref{eq-lem-Pk}. This completes the proof.
\end{proof}

\subsection{Proof of Proposition~\ref{prop-inexactcondition}}
\label{subsec-prop-inexactcondition}
\begin{proof}
	It is not difficult to show that
	\begin{equation*}
		(w^{k+1}, y^{k+1}, x^{k+1}) = (\Lambda_k + \sigma_k \mathcal{T}_{\tilde{l}})^{-1} \Lambda_k
		\left( \Lambda_k^{-1}(\sigma_k \nabla \Psi_k(w^{k+1}, y^{k+1}), 0) + (w^k,y^k,x^k) \right).
	\end{equation*}
	Note here that $ \Lambda_k^{-1} $ is well-defined since \blue{$(\sigma_k \nabla\Psi_k(w^{k+1}, y^{k+1}), 0)\in \mathcal{X} $} and $ \Lambda_k $ is positive definite \blue{over} $ \mathcal{ X} $. Then, by Lemma~\ref{lem-Pk}, we have
	\begin{eqnarray*}
		&  & \norm{(w^{k+1}, y^{k+1}, x^{k+1}) - P_k(w^k, y^k, x^k)}_{\Lambda_k} \\
		& = & \norm{(\Lambda_k + \sigma_k \mathcal{T}_{\tilde{l}})^{-1} \Lambda_k
			\left( \Lambda_k^{-1}(\sigma_k \nabla \Psi_k(w^{k+1}, y^{k+1}), 0) + (w^k,y^k,x^k) - (w^k,y^k,x^k) \right) }_{\Lambda_k} \\
		& \leq & \norm{ \Lambda_k^{-1}(\sigma_k \nabla \Psi_k(w^{k+1}, y^{k+1}), 0)  }_{\Lambda_k}  \quad \mbox{(since $ P_k $ is non-expansive) }\\
		& \leq & \frac{\sigma_k}{\min\{1, \sqrt{\tau_k}, \sqrt{\tau_k\blue{\lambda_+(Q)}}\}}\norm{\nabla \Psi_k(w^{k+1}, y^{k+1})},
	\end{eqnarray*}
	as desired.
\end{proof}

\bibliographystyle{ACM-Reference-Format}
\bibliography{QPPAL}

\end{document}